\documentclass[11pt]{article}

\usepackage[margin=1in]{geometry}
\usepackage{amsmath,amsthm,amssymb,amsfonts,float}
\usepackage{listings}

\makeatletter
\newcommand{\subjclass}[2][2010]{%
  \let\@oldtitle\@title%
  \gdef\@title{\@oldtitle\footnotetext{#1 \emph{Mathematics subject classification.} #2}}%
}
\makeatother

\usepackage[hidelinks]{hyperref}
\usepackage{color}

\usepackage{esint}

\usepackage{mathtools,enumitem,mathrsfs}

\usepackage[compact]{titlesec}

\usepackage{comment}

\mathtoolsset{showonlyrefs=true}
    
\newcommand{\Kd}{\mathcal{K}_\delta}

\newcommand{\eps}{\epsilon}

\newcommand{\leqc}{\lesssim}

\newcommand{\grad}{\nabla}

\newcommand{\abs}[1]{\left| #1 \right|}
\newcommand{\set}[1]{\left\{ #1 \right\}}
\newcommand{\brak}[1]{\left\langle #1 \right\rangle} 

\newcommand{\R}{\mathbb{R}}

\newcommand{\C}{\mathbb{C}}

\newcommand{\N}{\mathbb{N}}

\renewcommand{\S}{\mathbb{S}}

\newcommand{\dee}{\mathrm{d}}

\newcommand{\dt}{\dee t}

\newcommand{\dx}{\dee x}

\DeclareMathOperator{\Id}{\mathrm{Id}}

\DeclareMathOperator{\Span}{\mathrm{span}}

\usepackage{bbm}

\renewcommand{\P}{\mathbf{P}}

\newcommand{\E}{\mathbf{E}}
\newcommand{\B}{\mathbb{B}}

\newcommand{\PP}{\mathbf P}

\newcommand{\Kc}{\mathcal{K}}

\newcommand{\Bc}{\mathcal{B}}
\newcommand{\Ec}{\mathcal{E}}

\newtheorem{theorem}{Theorem}[section]
\newtheorem{proposition}[theorem]{Proposition}

\newtheorem{lemma}[theorem]{Lemma}
\newtheorem*{lemma*}{Lemma}
\newtheorem{claim}[theorem]{Claim}

\newtheorem{assumption}{Assumption}

\theoremstyle{definition}
\newtheorem{definition}[theorem]{Definition}
\newtheorem{remark}[theorem]{Remark}

\numberwithin{equation}{section}
    
\begin{document}

\title{Existence of stationary measures for partially damped SDEs with generic, Euler-type nonlinearities}
\author{Jacob Bedrossian \and Alex Blumenthal \and Keagan Callis \and Kyle Liss 
} 
\maketitle

\begin{abstract}
We study nonlinear energy transfer and the existence of stationary measures in a class of degenerately forced SDEs on $\R^d$ with a quadratic, conservative nonlinearity $B(x,x)$ constrained to possess various properties common to finite-dimensional fluid models and a linear damping term $-Ax$ that acts only on a proper subset of phase space in the sense that $\mathrm{dim}(\mathrm{ker}A) \gg 1$. Existence of a stationary measure is straightforward if $\mathrm{ker}A = \{0\}$, but when the kernel of $A$ is nontrivial a stationary measure can exist only if the nonlinearity transfers enough energy from the undamped modes to the damped modes. We develop a set of sufficient dynamical conditions on $B$ that guarantees the existence of a stationary measure and prove that they hold ``generically'' within our constraint class of nonlinearities provided that $\mathrm{dim}(\mathrm{ker}A) < 2d/3$ and the stochastic forcing acts directly on at least two degrees of freedom. We also show that the restriction $\mathrm{dim}(\mathrm{ker}A) < 2d/3$ can be removed if one allows the nonlinearity to change by a small amount at discrete times. In particular, for a Markov chain obtained by evolving our SDE on approximately unit random time intervals and slightly perturbing the nonlinearity within our constraint class at each timestep, we prove that there exists a stationary measure whenever just a single mode is damped.
\end{abstract}

\setcounter{tocdepth}{3}
{\small\tableofcontents}

\section{Introduction}
\label{sec:intro}

Many physical phenomena involve the nonlinear, conservative transfer of energy from weakly damped degrees of freedom driven by an external force to other modes that are more strongly damped. 
For example, in hydrodynamic turbulence, energy enters the system primarily at large spatial scales, but at high Reynolds number, dissipative effects are only significant at very high frequencies. Nevertheless, empirical observations suggest that the nonlinearity transfers energy to small scales at a rate that allows statistically stationary solutions to have bounded energy in the infinite Reynolds number limit, with the energy input balanced by a nontrivial flux of energy through arbitrarily small length scales. This is an instance of a phenomenon typically referred to as \textit{anomalous dissipation} and is one of the fundamental predictions of turbulence theory (see e.g. discussions in \cite{Frisch}). While energy cascades and dissipation anomalies have been satisfactorily studied in restricted settings such as passive scalar turbulence \cite{BBPSBatchelor}, linear shell models \cite{Mattingly07,Mattingly07Short}, and some simplified nonlinear models \cite{VicolGlattHoltzFriedlander}, an understanding of such phenomena in realistic, infinite dimensional nonlinear systems seems largely out of reach.
Motivated by discussions found in \cite{Majda16} and \cite{EHS15}, there is a natural (much simpler) analogue in stochastically forced, finite-dimensional models with fluid-like properties, namely, to determine under what conditions the system admits a stationary measure even if only a subset of the degrees of freedom are directly damped.
This can only be possible if the nonlinearity transfers enough energy from the undamped modes to the damped modes.

In this paper, we study the nonlinear transfer of energy and existence of stationary measures in a class of SDEs on $\R^d$ of the form
\begin{align}
\dee x_t = B(x_t,x_t) \dee t - Ax_t \dee t + \sum_{j=1}^d \sigma_j e_j \dee W_t^{(j)}, \label{eq:SDE}
\end{align}
where $\set{e_j}$ denote the canonical basis vectors, $\sigma_j \in \mathbb R$ are fixed coefficients (some of which are allowed to be zero), $A$ is a symmetric positive semi-definite matrix (with $\mathrm{dim}(\mathrm{ker}A) > 0$, representing the number of modes left undamped), and the $\set{W_t^{(j)}}$ are iid Brownian motions on the canonical stochastic basis $(\Omega,\PP,\mathcal{F}, \mathcal{F}_t)$. In many of the motivating infinite dimensional examples the forcing usually acts mainly on scales widely separated from those on which the dissipation is important. With this in mind,  we consider the general case in which many of the $\sigma_j$ may vanish (below we will assume only two coefficients are non-vanishing). Here, $B = B(x,x)$ is a bilinear vector field on $\R^d$ satisfying 
\begin{equation} \label{eq:energycons}
	x \cdot B(x,x) = 0 \quad \forall x \in \R^d
\end{equation}
along with some additional constraints to be specified below (see Definition~\ref{def:CCI}). The property \eqref{eq:energycons} implies that $B$ is conservative in the sense that solutions to the ODE $\dot{x} = B(x,x)$ conserve the energy $x \mapsto |x|^2$. Our choice of nonlinearity is primarily motivated by fluid mechanics, as most finite-dimensional approximations of the incompressible Euler equations yield such a bilinear nonlinearity. The class contains several other classical fluid mechanics models such as the shell models GOY \cite{YamadaOhkitani87,Gledzer1973} and Sabra \cite{SABRA}, in addition to the well-studied Lorenz 96 model \cite{Lorenz96}, used frequently as a standard benchmark for simulating chaotic dynamics. 

The first and fourth authors studied this problem recently in \cite{BedrossianLiss22}, finding a variety of sufficient conditions on $B$ under which \eqref{eq:SDE} would admit stationary measures. The conditions could be verified for a handful of examples with $\mathrm{dim}(\mathrm{ker}A)$ relatively small, but were not suitable for treating cases of \eqref{eq:SDE} with $\mathrm{dim}(\mathrm{ker}A) \gg 1$. In this paper, we make the following contributions. First, we provide fairly robust sufficient dynamical conditions for the existence of a stationary measure that are in principle applicable to a variety of examples when only a few modes are damped. Next, we show that for ``generic'' choices of  nonlinearity $B$, these conditions are satisfied provided that $\mathrm{dim}(\mathrm{ker}A) < 2d/3$. Lastly, with just a slight modification of the earlier proofs, we show that it suffices to damp just a single mode if one allows 
the nonlinearity to fluctuate slightly in time. In particular, for a Markov chain obtained by iterating \eqref{eq:SDE} on approximately unit random time intervals and slightly perturbing the nonlinearity within our constraint class at each timestep, we prove that there exists at least one stationary measure provided that $\mathrm{dim}(\mathrm{ker}A) < d$.

\subsection{Results for fixed, generic nonlinearities} \label{sec:fixedresults}

The following is the class of ``Euler-like'' bilinear vector fields we shall consider in this paper. 
\begin{definition}[Constraint class $\Bc$]  \label{def:CCI}
We say that a bilinear vector field $B$ belongs to the \emph{constraint class} $\Bc$ if the following conditions hold: 
\begin{align}
	x \cdot B(x,x) & = 0 \quad \text{ and } \label{eq:energyPres} \\ 
	\grad \cdot B(x,x)  & = 0 \quad \label{eq:divFree}
\end{align}
at all $x \in \R^d$; and 
\begin{align}
	B(e_j,e_j) = 0 \label{eq:vanishCoordAxes} 
\end{align}
for all $1 \leq j \leq d$. 
\end{definition}
Bilinear vector fields $B$ as in Definition \ref{def:CCI} can be represented in the form
\[B(x,y) = \begin{pmatrix}
	x^T B^1 y \\ \vdots \\ x^T B^d y 
\end{pmatrix} \, , \]
where $B^i$ is a $d\times d$ matrix with $(j,k)$ entry $b^i_{jk}$. Since we are evaluating $B(x,x)$ throughout, we lose no generality in assuming each $B^i$ is symmetric. In this way, we shall view \[ \Bc \subset \R^{d^3} \, , \] with a coefficient tensor $b = (b^i_{jk}) \in \Bc$ corresponding to the bilinear vector field $B = B_b$. 
It is straightforward to check that equations \eqref{eq:energyPres}, \eqref{eq:divFree} and \eqref{eq:vanishCoordAxes} are linear in the coefficient $b = (b^i_{jk}) \in \R^{d^3}$, hence $\Bc \subset \R^{d^3}$ is a linear subspace, carrying with it the natural topology and notion of Lebesgue measure.


Throughout, $A$ will be a positive semidefinite $d \times d$ matrix with
\begin{align*}
\Kc := \mathrm{ker} A = \Span\set{e_1,...,e_J}
\end{align*}
for some $J \leq d$. 
That is, modes $e_1, \dots, e_J$ are left undamped, and modes $e_{J+1}, \dots, e_d$ are damped. Our first main result is then stated as follows.
\begin{theorem} \label{thm:main} 
There is an open, dense, and full Lebesgue-measure subset $\mathring{\Bc} \subset \Bc$ with the following properties. 
Assume \[J  < \frac23 d \qquad\] and that \[\text{$\sigma_i \neq 0$ for at least two modes $i \in \set{1, \dots, d}$} \,. \] Then, for any $b \in \mathring{\Bc}$, equation \eqref{eq:SDE} with $B = B_b$ admits a
unique stationary measure $\mu$. This stationary measure is absolutely continuous with respect to Lebesgue measure, with smooth and strictly positive density, and satisfies the moment estimates
\begin{align} \label{eq:momentEstimates} \int_{\mathbb R^d} |x|^p \dee \mu(x) < \infty \end{align}
for any $p < \infty$. 
\end{theorem}
\begin{remark}
It is likely that the result holds all the way to $J < d$. Our proof provides a scheme that can in principle be used to show this, but requires verifying an algebraic condition that seems quite difficult to check. For more discussion on this, see Section~\ref{subsec:genericQuadPassthrough} and in particular Remark~\ref{rem:d/3}. 

 \end{remark}

\begin{remark}
Consider for definiteness the case that $\sigma_j \neq 0$ if and only if $j=1,2$ and that the hypotheses of Theorem \ref{thm:main} hold.
Then, by It\^o's lemma and the pointwise ergodic theorem, for $\mathbb P \times Leb$-a.e. $(\omega,x)$ there holds
\begin{align*}
\lim_{T \to \infty} \frac{1}{T} \int_0^T \brak{\Pi_{\Kc} x_t, \Pi_{\Kc}B(x_t,x_t) } \dee t = \frac{1}{2}(\sigma_1^2 + \sigma_2^2), 
\end{align*}
which is simply an expression of the non-vanishing energy flux from the undamped modes to the damped modes.
This also implies that the stationary measures cannot be close to Gaussian measures; see \cite{Majda16} for further discussions. 
\end{remark}

\subsection{Outline of the proof of Theorem \ref{thm:main}} \label{sec:outline}

We comment here primarily on the proof of existence of a stationary measure $\mu$. Uniqueness and other properties of $\mu$ follow via standard techniques and will be discussed briefly at the end of this section.

\subsubsection*{Overcoming partial damping}

Applying It\^{o}'s formula and using \eqref{eq:energyPres}, it can be shown that the solution of \eqref{eq:SDE} with initial condition $x_0 \in \R^d$ satisfies the energy estimate

\begin{equation} \label{eq:energyestimate}
	\E |x_t|^2 + 2 \int_0^t \E [Ax_s \cdot x_s] \dee s = |x_0|^2  +  t \sum_{i=1}^d \sigma_i^2 \, . 
\end{equation}
When $A$ is positive definite (i.e., $\Kc = \emptyset$), equation \eqref{eq:energyPres} and the estimate $Ax \cdot x \gtrsim |x|^2$ (valid for all $x \in \R^d$) implies that for any fixed initial $x_0 \in \R^d$, 
\[\sup_{t \ge 1} \frac{1}{t}\int_0^t \E |x_s|^2 \dee s < \infty \,. \] 
This bound implies that the time-averaged empirical measures 
\[\frac{1}{t} \int_0^t \E [\delta_{x_s}] \dee s\]
of the process $(x_t)$ are tight, which from a standard application of the Krylov-Bogoliubov argument (see, e.g., \cite[Theorem 1.1]{sinai1989dynamical}) implies existence of a stationary measure $\mu$ for $(x_t)$ with finite second moments. 

On the other hand, in the partial damping setting $\Kc \neq \emptyset$, the estimate $Ax \cdot x \gtrsim |x|^2$ is false along $x \in \Kc$ and the argument breaks down. However, one can still hope to recover some control if, roughly speaking, typical trajectories $(x_t)$ do not spend too much time near $\Kc$. Indeed, it is not hard to show that one can carry out the Krylov-Bogoliubov argument as intended under the weaker time-averaged condition 
\begin{equation} \label{eq:O(1)average}
	1+\E \int_0^T Ax_s \cdot x_s \dee s \gtrsim \E \int_0^T |x_s|^{2r}\dee s \quad \forall x_0 \in \R^d \, ,  
\end{equation}
where $T \approx 1$ and $r \in (0,1]$ are constants; see \cite[Lemma 2.1]{BedrossianLiss22} for details.


To this end, the paper \cite{BedrossianLiss22} showed that one can reduce \eqref{eq:O(1)average} to studying time-averaged coercivity estimates for the dissipation over short, $x_0$-dependent timescales, and for initial data at sufficiently high energy and close to the undamped region $\Kc$. Theorem~\ref{thrm:abstract} below is a version of the main abstract criteria for existence from \cite{BedrossianLiss22}. In all that follows, $\Pi_\Kc$ denotes the orthogonal projection onto $\Kc$, and $\Pi^\perp_\Kc = \Id - \Pi_\Kc$. 

\begin{theorem} \label{thrm:abstract}
	Let $T:[1,\infty) 
	\to (0,\infty)$ and $F:[1,\infty)\to (0,\infty)$ be functions satisfying $$\lim_{E\to \infty} T(E) = 0 \quad \text{and} \quad \liminf_{E \to \infty}E^{-2r}F(E) > 0$$ for every $r \in (0,1)$. 
	Suppose that there exist positive constants $c_0,C_0, E_0$, and $\delta_0$ such that for every $E \ge E_0$ and initial condition 
	\[ x_0 \in \{x \in \R^d: ||x| - E| \le \delta_0 E \text{ and } |\Pi_{\Kc}^\perp x| \le \delta_0 E \} \]
	the associated solution $(x_t)$ of \eqref{eq:SDE} satisfies the time-averaged estimate
	\begin{equation} \label{eq:coercivity}
		\frac{1}{T(E)}\E \int_0^{C_0 T(E)} A x_t \cdot x_t \dt \ge c_0 F(E). 
		\end{equation}
	Then, there exists an invariant measure for \eqref{eq:SDE} that has polynomial moments of all orders.
\end{theorem}

\begin{remark}
What is shown in \cite{BedrossianLiss22} is actually that if the hypotheses above hold, but with some $F$ that satisfies $\liminf_{E \to \infty} E^{-2r}F(E) > 0$ for just some fixed $r < 1$, then \eqref{eq:SDE} admits an invariant measure $\mu$ with $\int_{\R^d} |x|^p \mu(dx) < \infty$ for every $p < 2r/(1-r)$ (see in particular \cite[Lemmas 2.1 and 2.2]{BedrossianLiss22}). In other words, $F$ can grow much slower than assumed in \eqref{eq:coercivity} if the goal is just to deduce existence. 
Theorem~\ref{thrm:abstract} is a simple corollary of this result that gives a criteria for the invariant measure constructed to have polynomial moments of all orders and will be natural to apply in our setting.
\end{remark}

\begin{remark}
The condition \eqref{eq:coercivity} is close to obtaining quantitative exit time estimates from the set $\set{x \in \R^d: \abs{\Pi_{\Kc}^\perp x} \leq \delta_0\abs{\Pi_{\Kc} x} \approx \delta_0 E}$ that vanish at least like $T(E)$ as $E \to \infty$ for initial conditions starting on or near $\Kc$ (i.e. quantitative estimates on how quickly solutions escape from neighborhoods of $\Kc$).  
\end{remark}

\subsubsection*{Dynamical conditions to apply Theorem \ref{thrm:abstract}}

The plan for applying Theorem \ref{thrm:abstract} is as follows: first, to develop a set of sufficient \emph{dynamical} conditions on the dynamics of the deterministic flow $\dot x = B(x,x), B = B_b$ which allow to apply Theorem \ref{thrm:abstract}; and second, to show that these dynamical conditions hold for a `generic' set of $b \in \Bc$. 

The dynamical conditions we  impose are as follows. Here and in all that follows we will write 
\begin{align}
	\Ec = \bigcup_{i = 1}^d \Span(e_i)
\end{align}
for the union of the coordinate axes of $\R^d$, corresponding to the required equilibria of $B_b$ in the constraint class as in \eqref{eq:vanishCoordAxes}.

\begin{itemize}
	\item[(1)] (Hyperbolicity) The linearization $DB(x,x)$ at each equilibrium $x \in \Ec \setminus \{0\}$ admits a single eigenvalue along the imaginary axis. 
	\item[(2)] (Hypoellipticity) For any $d \times d$ matrix $M$ and $\epsilon \in [0,1]$, the collection of vector fields $\{B = B_b + \epsilon Mx, \sigma_1 e_1, \ldots, \sigma_d e_d\}$ satisfies the parabolic H\"ormander condition.\footnote{Let $\{X_j\}_{j=0}^m$ be a collection of smooth vector fields on $\R^d$. Define $V_0 = \{X_j\}_{j=1}^m$ and then for $k \ge 1$ let 
		$$V_k = V_{k-1} \cup \{[X_j,Y]: 0 \le j \le m, \quad Y \in V_{k-1}\}.$$  
		Recall that $\{X_j\}_{j=0}^m$ is said to satisfy the \textit{parabolic H\"{o}rmander condition} if for every $x \in \R^d$ there exists some $N \in \N$ such that $\mathrm{span} (V_N) = \R^d$.}
	\item[(3)] (Dynamics on $\Kc$) For any solution $x(t)$ to the deterministic problem $\dot x = B(x,x)$, it holds that if $x(t) \in \mathring{\Kc} := \Kc \setminus \Ec$, then
	\[\frac{d^j}{dt^j} x(t) \notin \mathcal{\Kc} \]
	for some $j \geq 1$. 
\end{itemize}

The following is proved in Section \ref{sec:existsStatMeas}. 

\begin{proposition}\label{prop:dynCondImpliesExist1}
	Assume $b \in \Bc$ is such that (1) -- (3) above hold for the bilinear vector field $B = B_b$. Then, the hypotheses of Theorem \ref{thrm:abstract} hold with $$T(E) = \frac{\log E}{E} \quad \text{ and } \quad F(E) = \frac{E^2}{\log E},$$ and in particular there exists a stationary measure for \eqref{eq:SDE} satisfying the moment estimate \eqref{eq:momentEstimates} for all $p \geq 1$. 
\end{proposition}

\begin{remark} \  
	\begin{itemize}
\item[(a)] For the deterministic flow $\dot x = B(x,x)$, assumption (1) ensures trajectories initiated off of stable manifolds through $\Ec$ are repelled away from $\Ec$ exponentially fast. Assumption (2) is a standard tool in stochastic analysis, used to check that the transition kernels for the Markov process $(x_t)$ have smooth densities with respect to Lebesgue measure \cite{H67}. In our setting, hypoellipticity will be used to ensure that with high probability, enough noise is injected so that  trajectories $(x_t)$ of the SDE \eqref{eq:SDE} avoid the stable manifolds through $\Ec$. Finally, assumption (3) is used to ensure trajectories near $\mathring{\Kc}$ do not linger there for too long, and in particular 
precludes the existence of invariant sets for the deterministic flow within the rest of the undamped set $\mathring{\Kc}$. Notice, for example, that this assumption implies that the stable manifolds through $\Ec$ intersect $\Kc$ on sets of positive codimension. 
\item[(b)] For high-energy initial conditions and at the timescales considered in our application of Theorem \ref{thrm:abstract}, the noise and dissipation affecting $(x_t)$ are actually quite weak relative to the strength of the nonlinearity $B = B_b$. This is essentially why it suffices to impose dynamical conditions on $B_b$.
On the other hand, this is in tension with how the noise must be used in a critical way to avoid the stable manifolds of $\Ec$. 
		Dealing with the degeneracy of the noise and its effective weakness at high energy in the proof of Proposition~\ref{prop:dynCondImpliesExist1} requires a slightly quantitative hypoelliptic smoothing estimate and crucially relies on the exponential (instead of algebraic-in-time) instability of $\Ec$. 
		\item[(c)] With $T(E)$ and $F(E)$ as defined in the statement of Proposition~\ref{prop:dynCondImpliesExist1}, the estimate \eqref{eq:coercivity} says essentially that solutions of \eqref{eq:SDE} with $|x_t| \approx E$ spend roughly $(\log E)^{-1}$ fraction of their time in regions where $|\Pi_{\Kc}^\perp x_t|^2 \approx |x_t|^2$. This is most difficult to establish for initial conditions very close to $\Ec$ and indeed obtaining a quantitative estimate on the escape time of solutions from a suitable neighborhood of $\Ec$ is the most technical step in the proof of Proposition~\ref{prop:dynCondImpliesExist1}. Our choice of the particular timescale $T(E) = E^{-1}\log E$ comes from the scaling of the typical escape time for \eqref{eq:SDE} from the vicinity of a spectrally unstable fixed point in the high-energy limit. For more discussion on the exit time bound that we require and how it fits into the existing literature, see the beginning of Section~\ref{sec:exittimes}.
		\item[(d)] It is likely that Proposition~\ref{prop:dynCondImpliesExist1} can still be established with assumption (1) weakened to the spectral instability of $\Ec\setminus \{0\}$ (i.e., to allow $DB(x,x)$ to have a more general center subspace) or even when $B$ possesses a more general manifold of unstable equilibria contained in $\mathrm{ker}A$. Such generalizations could become relevant if more constraints are put on $B$ (e.g., a second conservation law), but we did not need to pursue them in the present paper. 
		
	\end{itemize}
\end{remark}

To complete the proof of the existence portion of Theorem \ref{thm:main}, it will suffice now to check that assumptions (1) -- (3) are `generic'. The following is proved in Section \ref{sec:genericConditions1}. 

\begin{proposition} \label{prop:genericconditions}
	There is an open, dense, and full Lebesgue-measure subset $\mathring{\Bc} \subset \Bc$ with the following property. 
	Assume $J < 2d/3$ and that $\sigma_i \neq 0$ for at least two indices $i \in \{ 1, \dots, d\}$. Then, dynamical conditions (1) -- (3) as above hold for all $B = B_b, b \in \mathring{\Bc}$. 
\end{proposition}

\begin{remark}
The most challenging part of proving Proposition~\ref{prop:genericconditions} is establishing that condition (3) holds generically under the assumption $J < 2d/3$. For this, we rely on a suitable application of the Transversality Theorem (see Theorem~\ref{thm:transversality} and Section~\ref{subsec:genericQuadPassthrough} for details).
\end{remark}

With the existence portion of Theorem~\ref{thm:main} established, the remaining claims follow from well-known arguments. Smoothness of the Markov transition kernel follows from hypoellipticity, while positivity of the Markov transition kernel follows from the geometric control theory arguments found e.g. in \cite{HerzogMattPositivity,HMGHSaturation} (together with the conditions in Definition \ref{def:CCI} and the parabolic H\"ormander condition). 
The uniqueness of the measure then finally follows from Doob-Khasminskii \cite{DPZErgodicity}. 

\subsection{Results for generic nonlinearities with time switching} \label{sec:introswitching}

In the proof of Proposition~\ref{prop:genericconditions}, the restriction $J < 2d/3$ is used only to establish the dynamical condition (3). That is, to show that for $b \in \mathring{\Bc}$ the deterministic problem $\dot{x} = B_b(x,x)$ does not possess any invariant sets in $\Kc$ other than the fixed points $\Kc \cap \Ec$. Regardless of $J$, it is reasonable to expect that any \textit{particular} trajectory for $B_b$ that remains in $\Kc \setminus \Ec$ for all time should not persist under typical perturbations of the coefficient $b$. This intuition comes in part from the easily verified fact that for any $J < d$ and fixed $x \in \Kc \setminus \Ec$, the Jacobian of the map $b \mapsto B_b(x,x)$ evaluated at any $b_0 \in \Bc$ contains $\Kc^\perp$ in its range. This discussion suggests that the restriction $J < 2d/3$ in Theorem~\ref{thm:main} can be weakened if one modifies the model to allow the choice of generic coefficient to fluctuate in some way temporally. We show that this is indeed the case for a Markov chain obtained by running \eqref{eq:SDE} on roughly unit time intervals and changing the coefficient that defines $B$ slightly at each timestep. 

Informally, the Markov chain we consider is defined as follows. Fix coefficients $b_* \in \mathring{\Bc}$ and let $\mathcal{S} \subseteq \mathring{\Bc}$ be any open ball that contains $b_*$ and is compactly contained in $\mathring{\Bc}$ (recall that $\mathring{\Bc}$ is open). Let $I = [1/2,3/2]$. The first step in the chain is defined by picking an element $(b_1,t_1)$ from the uniform measure on $\mathcal{S} \times I$ and running \eqref{eq:SDE} for a time $t_1$ with $B = B_{b_1}$. The second step is obtained by independently sampling another pair $(b_2, t_2) \in \mathcal{S} \times I$ and running \eqref{eq:SDE} with $B= B_{b_2}$ for a time $t_2$. This process continues, with a new choice of coefficient and runtime for \eqref{eq:SDE} being chosen at each step. 

To define precisely the discrete-time system described above, let us assume here that $\Omega = C_0([0,\infty);\R)^{\otimes d}$ is the $d$-fold product of the classical Wiener spaces equipped with the associated Borel $\sigma$-algebra and Wiener measure $\P$. Let $\theta_t: \Omega \to \Omega$ denote the left shift by $t \ge 0$ (i.e., $\theta_t \omega(s) = \omega(s+t) - \omega(t)$) and recall that $\theta_t$ leaves $\P$ invariant. Since we work with additive noise, for $\omega \in \Omega$ and $b \in \Bc$ the SDE \eqref{eq:SDE} with $B = B_b$ has a well-defined random flow map $(t,x) \mapsto \varphi_{(\omega, b)}^t(x)$. In the definition of $\varphi_{(\omega,b)}^t$, the same damping matrix $A$ with $\Kc = \mathrm{span}\{e_1,\ldots, e_J\}$ and noise coefficients $\{\sigma_i\}_{i=1}^d$ with at least two nonzero are fixed and used for every $b \in \mathcal{S}$.
We now augment the probability space to allow for the random time intervals and perturbations in the nonlinearity.  Let $m_\mathcal{S}$ and $m_I$ be the normalized Lebesgue measures on the Borel $\sigma$-algebras of $\mathcal{S}$ and $I$, respectively. Then, define the probability space $(\hat{\Omega}, \hat{\P}, \hat{\mathcal{F}})$, where $\hat{\Omega} = \Omega \times \mathcal{S}^\N \times I^\N$, $\hat{\P} = \P \times m_\mathcal{S}^\N \times m_I^\N$, and $\hat{\mathcal{F}}$ is the associated product $\sigma$-algebra. For $\hat{\omega} = (\omega, \{b_i\}_{i=1}^\infty, \{t_i\}_{i=1}^\infty) \in \hat{\Omega}$ and $x \in \R^d$, we define 
\begin{equation} \label{eq:ChainDefinition}
	\Phi_n(\hat{\omega},x) = \varphi_{(\theta_{\tau_{n-1}}\omega, b_n)}^{t_n} \circ \varphi_{(\theta_{\tau_{n-2}}\omega, b_{n-1})}^{t_{n-1}} \circ \ldots \circ \varphi_{(\omega,b_1)}^{t_1}(x),
\end{equation}
where $\tau_k= \sum_{i=1}^k t_k$. It is easy to check that $\{\Phi_n\}_{n\in \N}$ defines an $\R^d$-valued Markov chain over the probability space $(\hat{\Omega}, \hat{\P}, \hat{\mathcal{F}})$.

Our second main result, proven in Section~\ref{sec:timedependent}, concerns the existence of an invariant measure for $\{\Phi_n\}_{n \in \N}$ when just a single mode is damped. 
\begin{theorem} \label{thm:mainswitched}
	Let $J < d$. For any $b_* \in \mathring{\Bc}$ the Markov chain $\{\Phi_n\}_{n\in \N}$ defined above has at least one stationary distribution.
\end{theorem}

\begin{remark} \
	\begin{itemize}
		\item[(a)] The proof of Theorem~\ref{thm:mainswitched} is mostly a corollary of the arguments needed to obtain Theorem~\ref{thm:main}. The key new ingredient is Lemma~\ref{lem:SwitchingTransverse}, which provides a modified version of the dynamical assumption (3)  (holding for any $J < d$) and makes rigorous the intuition given at the beginning of this section.
		\item[(b)] Switching the nonlinearity on random, rather than fixed, time intervals is not related to adapting assumption (3). Instead, it is used just to ensure that the switches occur with sufficiently high probability when trajectories are not too close to $\Ec$. It is likely that the switching times could be made deterministic at the cost of complicating Lemma~\ref{lem:KdeltaTime}, but for the sake of simplicity we did not pursue this.
		\item [(c)] Given the assumption $\lim_{E \to \infty} T(E) = 0$ in Theorem~\ref{thrm:abstract} and that the natural timescale of $\dot{x} = B(x,x)$ is $E^{-1}$ when $|x| = E$, it might be surprising that in Theorem~\ref{thm:mainswitched} we are able to switch the nonlinearity on approximately unit time intervals that do not depend on $|x|$. This is because while we will only extract dissipation on a small $E^{-1}$ portion of the time interval following each switch, when $|\Pi_{\Kc}^\perp x_t| \approx E$ this is more than enough to outweigh the energy input by the noise. It is true, however, that switching the nonlinearity on shorter timescales at high energy would allow us to prove sharper estimates on the stationary measure (e.g., moment bounds closer to or matching those in Theorem~\ref{thrm:abstract}). 
	\end{itemize}
\end{remark}

\subsection{Previous work and discussion}

As mentioned above, the first and fourth authors already considered the problem of constructing a stationary measure for \eqref{eq:SDE} in \cite{BedrossianLiss22}. They were able to develop sufficient conditions implying the hypothesis of a generalized Theorem~\ref{thrm:abstract} that could be checked in examples including the Galerkin Navier-Stokes equations on a period box, Lorenz 96, and the Sabra shell model, but in all cases restricted to when $\mathrm{dim}(\mathrm{ker}A)$ is relatively small. There are various earlier works besides \cite{BedrossianLiss22} that have considered the existence of stationary measures for examples of \eqref{eq:SDE} and related partially damped systems. Those that consider settings most similar to the one here are \cite{CamrudThesis, FoldesHerzogGlatt21, BrendanThesis, HerzogMatt24}. The works \cite{CamrudThesis, FoldesHerzogGlatt21} both prove results for low-dimensional models, specifically the existence of stationary measures for Lorenz 96 in 4d with two damped modes \cite{CamrudThesis} and the Lorenz 63 model with a single axis of unstable fixed points left undamped \cite{FoldesHerzogGlatt21}. The previous work perhaps most related to ours is \cite{BrendanThesis}, which considers exactly \eqref{eq:SDE} in general dimension, but with $B$ assumed just to satisfy \eqref{eq:energyPres}. In the same spirit as our Proposition~\ref{prop:dynCondImpliesExist1}, the author proves a general existence theorem under a set of assumptions that include the only deterministic trajectories remaining in the undamped region for all time consisting precisely of spectrally unstable fixed points of $B$ (i.e., a version of our assumptions (1) and (3)). However, in addition to using elliptic noise, the result assumes a special structure of the polynomials 
$$P_{j+1}(x):=\frac{d^j}{dt^j}x(t)|_{t = 0}, \quad x(0)= x \not \in \Kc\cap \Ec$$
that is much stronger than our dynamical assumption (3) and not true in typical examples.\footnote{Specifically, it is assumed that there is some $n_* \in \N$ such that for every $x \not \in \Kc \cap \Ec$ there exists $n \le n_*$ for which $|\Pi_{\Kc}^\perp P_{n+1}(x)| \gtrsim |x|^n d(x,\Kc \cap \Ec)$. This is easily seen to be false, for instance, in typical shell models, even in cases where our assumption (3) can be checked quickly by hand. This modified version of (3) was required in \cite{BrendanThesis} primarily to compensate for an analysis of the dynamics in a neighborhood of the undamped fixed points that was based purely on approximating by the linearization.} Together with the practical challenge of proving the instability of equilibria in fixed models, the sufficient conditions in \cite{BrendanThesis} were not verified in any high-dimensional examples. On the other hand, the general result in \cite{BrendanThesis} does allow for a manifold of undamped equilibria much more complicated than just the coordinate axes. 

The very recent paper \cite{HerzogMatt24} studied partial damping in the setting of the ``randomly split'' models introduced in \cite{MattinglySplit1} and provided the first examples of fluid-like systems where existence of a stationary measure could be established with a fixed number of modes damped and dimension arbitrary. These models split the conservative vector field into a sum of simpler interactions (often each giving rise to an exactly solvable or integrable system) that respect the systems conservation laws and can be thought of as natural building blocks of the original dynamics. The associated stochastic system is then defined by cycling through the flow of each individual interaction of the splitting over random, independent time intervals. Existence of a stationary measure for random splittings of Galerkin Euler and Lorenz 96 is proven in \cite{HerzogMatt24} when just a few degrees of freedom are damped. The random splittings and setting of our Theorem~\ref{thm:mainswitched} share some similarities in that they both involve a Markov chain obtained by composing dynamics with a different conservative vector field on each timestep, but how this property is utilized is different in the two cases. The proofs in \cite{HerzogMatt24} that sufficient energy transfers to the damped modes use the solvability of the individual interactions and leverage some rare stochastic realizations allowed because of the splitting to simplify the overall dynamics. This should be contrasted with the ``random switching'' in our Theorem~\ref{thm:mainswitched}. The flows that we compose to form our Markov chain $\Phi_n$ involve the full evolution of \eqref{eq:SDE} and hence do not simplify the dynamics of any given $B_b$, but cycling between different nonlinearities is used crucially to break possible invariant sets in $\Kc$ when only a single mode is damped.


The study of partial dissipation and related problems also extends outside of \eqref{eq:SDE} and its variants. For example,  partially damped and hypoelliptic equations arise naturally in Langevin dynamics and related Hamiltonian systems, where forcing and dissipation act only on the momentum variable. Exponential convergence to equilibrium was proven in \cite{HerzogMatt19} for the Langevin SDE of many particles interacting via singular potentials by constructing a nontrivial Lyapunov function.\footnote{Note that in this setting there is an exact formula for the invariant Gibbs measure, and so existence is trivial.} Other results for Hamiltonian systems include  \cite{CuneoEckmannHairer18,HairerMattchain09}, which establish the existence of a stationary measure for systems of anharmonic coupled oscillators that interact with Langevin heat baths through select modes (see also the related work \cite{Eckmann19}). Another body of work connected to partial damping considers stabilization by noise for systems with deterministic trajectories that exhibit finite-time blow up; see e.g. \cite{AthreyaMattingly, HerzogMatt15I,HerzogMatt15II,HerzogWehr12}. In these papers, existence of a unique stationary distribution is proven despite the underlying deterministic trajectories not always being globally defined. Even when all directions are damped, establishing such a result shares many similarities with the problem considered here in that one must show that trajectories escape the regions of phase space that exhibit finite-time blow up sufficiently fast.


\vspace{0.2cm}

\noindent \textbf{Acknowledgments}: The authors thank Sam Punshon-Smith for many useful discussions in the early stages of this project. 

This material is based upon work supported by the National Science Foundation under Grant Nos. DMS-2038056 (author J.B.); DMS-2009431 and DMS-2237360 (author A.B.), and DMS-2108633 (author K.L.). 


\section{Existence of stationary measures}
\label{sec:existsStatMeas}

The plan in this section is to prove Proposition~\ref{prop:dynCondImpliesExist1}. Throughout, $B$ denotes a bilinear vector field in the constraint class $\Bc$ that satisfies the dynamical assumptions (1) -- (3) introduced in Section~\ref{sec:outline}. We will begin in Section~\ref{sec:rescaled} by introducing a rescaling of \eqref{eq:SDE} that is natural to use for high-energy initial conditions and which we will employ in our analysis for the remainder of the section. Then, in Section~\ref{sec:exittimes} we prove an estimate on the exit time of solutions to \eqref{eq:SDE} from the vicinity of the unstable fixed points in $\Ec$. This will be the main technical step in checking the hypotheses in Theorem~\ref{thrm:abstract}.  The proof of Proposition~\ref{prop:dynCondImpliesExist1} using the exit time estimates is finally completed in Section~\ref{sec:finishproof}.

\subsection{High-energy rescaling of SDE} \label{sec:rescaled}

Recall that to prove Proposition~\ref{prop:dynCondImpliesExist1} it is sufficient to verify the hypotheses of Theorem~\ref{thrm:abstract} with 
\begin{equation} \label{eq:timescale}
	T(E) = \frac{\log E}{E} \quad \text{ and } \quad F(E) = \frac{E^2}{\log E}.
\end{equation}
As we need to prove \eqref{eq:coercivity} only for $E \gg 1$, it will be more convenient to work with a high-energy rescaling of \eqref{eq:SDE}. Fix $x _0 \in \R^n$ and let $|x_0| = E \gg 1$. If $x_t$ solves \eqref{eq:SDE}, then $\frac{1}{E} x_{t/E}$ is equal in law to the solution $\bar{x}_t$ of 
$$
		\dee \bar{x}_t = B( \bar{x}_t,  \bar{x}_t)\dt - \frac{1}{E} A  \bar{x}_t \dt + \frac{1}{E^{3/2}}\sum_{j=1}^d \sigma_j e_j \dee W_t^{(j)}
$$
with an initial condition  $\bar{x}_0$ lying on the unit sphere $\S^{d-1}$. Therefore, for $\epsilon \in (0,1)$ we introduce the following rescaling of \eqref{eq:SDE}:
\begin{equation} \label{eq:scaledSDE}
	\begin{cases}
		\dee x^\epsilon_t = B(x^\epsilon_t, x^\epsilon_t)\dt - \epsilon A x^\epsilon_t \dt + \epsilon^{3/2}\sum_{j=1}^r \sigma_j e_j \dee W_t^{(j)}, \\
		x^\epsilon_t|_{t=0} = x_0.
	\end{cases}
\end{equation}
It is straightforward to check that in order to verify \eqref{eq:coercivity} with the definitions \eqref{eq:timescale}, it suffices to show that there are constants $\delta_0, \epsilon_0,c_0,C_0 > 0$ such that for every $\epsilon \in (0,\epsilon_0)$ and $x_0 \in \S^{d-1}$ with $|\Pi_\Kc^\perp x_0| \le \delta_0$ we have 
\begin{equation} \label{eq:scaledcoercivity}
	\E \int_0^{C_0|\log \epsilon|} A x^\epsilon_t \cdot x^\epsilon_t \dt \ge c_0.
\end{equation}
Our goal in the remainder of this section is thus to prove \eqref{eq:scaledcoercivity}.

Before proceeding we establish some notation. We denote the Markov transition kernel associated with \eqref{eq:scaledSDE} by $P_t^\epsilon(x,\cdot)$, defined for $t > 0$, $x \in \R^d$, and a Borel subset $U \subseteq \R^d$ by
$$ P_t^\epsilon(x,U) = \P(x_t^\epsilon \in U| x_0 = x).$$
The associated Markov semigroup $P_t^\epsilon: B_b(\R^d) \to B_b(\R^d)$, which acts on the space $B_b(\R^d)$ of bounded, Borel measurable functions, is defined for $f \in B_b(\R^d)$ by
$$P_t^\epsilon f(x) = \int_{\R^d} f(x')P_t^\epsilon(x,\dee x') = \E (f(x^\epsilon_t)|x_0 = x).$$
Note that by assumption (2), for any initial condition $x_0 \in \R^d$ and $t > 0$, the law of \eqref{eq:scaledSDE} has a smooth density with respect to Lebesgue measure. We will denote this density by $p_\epsilon(t,x_0,x)$. That is, $p_\epsilon$ is such that  $P_t^\epsilon(x_0,\dx) = p_\epsilon(t,x_0,x)\dx$. Lastly, throughout this entire section, we write $a \leqc b$ if $a \le Cb$ for a constant $C$ that is independent of $\epsilon \in (0,1)$ and any other relevant parameters.

\subsection{Quantitative exit time estimates for hypoelliptic diffusions} \label{sec:exittimes}

The  constraint class $\Bc$ guarantees the existence of fixed points of $B$ that lie in $\Kc$, namely, the elements of $\cup_{i=1}^J \mathrm{span}(e_i) \subseteq \Ec$. In order to verify \eqref{eq:scaledcoercivity} for an initial condition $x_0$ nearby such a fixed point $\bar{x}$, we will require an estimate on the exit time of the rescaled process \eqref{eq:scaledSDE} from a small neighborhood of the unstable equilibria comprising $\Ec$. For $\delta \in (0,1)$ define the compact set
\begin{equation} \Kd = \{x \in \R^d:  \mathrm{dist}(x,\Ec) \ge \delta \text{ and } 1/2 \le |x| \le 3/2\} 
\end{equation}
Then, for $x_0 \in \R^d$ we define the stopping time 
\begin{equation}
	\tau^\epsilon_\delta(x_0,\omega) = \inf\{t \ge 0: x^\epsilon_t(\omega) \in \Kd\},
\end{equation}
where $x^\epsilon_t$ denotes the solution of \eqref{eq:scaledSDE} with initial condition $x_0$. The main result of this section is the following lemma.


\begin{lemma} \label{lem:scaledexit}
There exist constants $C,c, \delta > 0$ such that for all $\epsilon$ sufficiently small and $x_0 \in \S^{d-1}\setminus \Kd$ we have 
	\begin{equation} \label{eq:stoppingtimebound}
		\P(\tau^\epsilon_\delta(x_0,\omega) \le C|\log \epsilon|) \ge c.
	\end{equation}
\end{lemma}

The asymptotic properties in the small noise limit of the exit time and location of a diffusion process from the vicinity of an unstable, hyperbolic fixed point have been examined previously in some detail \cite{Kifer1981, Bakhtin2008, BakhtinNormalforms}. Heuristically, the fact that $|\log \epsilon|$ is both the natural and optimal scaling in $\epsilon$ for \eqref{eq:stoppingtimebound} to hold with $c$ independent of $\epsilon$ and $x_0$ is fairly clear. Indeed, provided that the diffusion is at least hypoelliptic, typical trajectories that begin on the stable manifold of $\bar{x}$ experience unstable perturbations on the order of $\epsilon^p$ for some positive power $p$. Therefore, due to the exponential instability, one expects there to be some $p,\lambda > 0$ such that $e^{\lambda \tau_\delta^\epsilon(x_0)} \epsilon^p \leqc 1$ with high probability and for the corresponding lower bound to also hold for certain choices of $x_0$. The previous works mentioned above provide estimates more precise than just capturing the logarithmic scaling in $\epsilon$, but they either assume that the diffusion is uniformly elliptic and the initial condition is exactly on the stable manifold \cite{Kifer1981, Bakhtin2008}, or that the diffusion is two-dimensional and the initial condition is already given a small unstable perturbation \cite{BakhtinNormalforms}. In proving Lemma~\ref{lem:scaledexit}, we must deal with the facts that the diffusion is merely hypoelliptic, the initial condition is entirely general, and each fixed point in $\Ec$ has a center manifold consisting of the radial direction.

The general structure of the proof of Lemma~\ref{lem:scaledexit} is to first construct a local-in-time (on the $|\log \epsilon|$ timescale), random center-stable manifold in a neighborhood of each fixed point in $\Ec \cap \S^{d-1}$. We then show that the unstable component of any random trajectory that starts off of this manifold grows exponentially fast using a straightforward argument based on the preservation of unstable cones. We conclude the proof by combining the steps above with a hypoelliptic smoothing estimate which implies that typical random trajectories quickly find themselves outside of an $\epsilon^p$-neighborhood ($p \gg 1$) of the random manifold, and hence escape on a $|\log \epsilon|$ timescale. 

\subsubsection{Hypoellipticity preliminaries}

We begin by establishing the needed quantitative-in-$\epsilon$ smoothing estimates. The bounds that we require are not especially precise, in the  sense that it is not necessary to capture the optimal regularization in each direction. Instead, we will just need an $L^1 \to L^2$ smoothing estimate that scales polynomially in $\epsilon$. As such, the estimates in this section are relatively straightforward  and probably clear to experts, but to our knowledge cannot be obtained as an immediate corollary of any estimates in the literature. Recall that we write $p_\epsilon(t,x_0,x)$ for the Lebesgue density of the law of $x_t^\epsilon$. The smoothing estimate we will require is given by lemma below.
\begin{lemma} \label{lem:densitysmoothing}
	There exist $C,p > 0$ such that for any $\epsilon \in (0,1)$ and $|x_0| \le 1$ we have
	\begin{equation}
		\|p_\epsilon(1,x_0,\cdot)\|_{L^2} \le C \epsilon^{-p}.
	\end{equation}
\end{lemma}


Our proof of Lemma~\ref{lem:densitysmoothing} will use a functional inequality that follows from \cite{H67}. To state it, we first need to introduce some notation. Let $B_R = \{x \in \R^d: |x| < R\}$  and for an open subset $D$ of a Euclidean space, let $\mathcal{T}(D)$ denote the collection of smooth vector fields defined on $D$. In what follows we identify vector fields with differential operators. That is, for $X \in \mathcal{T}(\R^d)$ and $g:\R^d \to \R$ we write $Xg$ to mean $X \cdot \grad g$. Let $\{X_j\}_{j=1}^m \subseteq \mathcal{T}(\R^d)$. For $f \in C_0^\infty((1/4,2)\times B_2)$, we define the H\"{o}rmander norm
$$\|f\|^2_{\mathscr{X}} = \int_0^\infty \int_{\R^d} |f(t,x)|^2 \dx \dt + \sum_{j=1}^m \int_0^\infty \|X_j f(t,\cdot)\|_{L^2(\R^d)}^2 \dt$$
and the associated dual norm 
$$ \|f\|_{\mathscr{X}^*} = \sup_{\varphi \in C_0^\infty((1/4,2)\times B_2), \|\varphi\|_{\mathscr{X}} \le 1} \int_0^\infty \int_{\R^d} \varphi(t,x)  f(t,x)\dx \dt. $$
The functional inequality below follows from a careful reading of \cite{H67}.

\begin{lemma} \label{lem:generalhormander}
	Let $X, \tilde{X} \in \mathcal{T}(\R^d)$ and for $\epsilon \ge 0$ define $X_{0,\epsilon} = X + \epsilon \tilde{X}$. Suppose that $\{X_j\}_{j=1}^m \subseteq \mathcal{T}(\R^d)$ is such that the collection $\{X_{0,\epsilon}, X_1, \ldots, X_m\} \subseteq \mathcal{T}(\R^d)$ satisfies the parabolic H\"{o}rmander condition for every $\epsilon \in [0,1]$. Then, there exist $s,p, C > 0$ such that for all $\epsilon \in (0,1)$ and $f \in C_0^\infty((1/4,2)\times B_2)$ we have 
	\begin{equation}
		\|f\|_{H^s(\R^{d+1})} \le  C \epsilon^{-p}\left(\|f\|_{\mathscr{X}} + \|(\partial_t - X_{0,\epsilon})f\|_{\mathscr{X}^*}\right).
	\end{equation}
\end{lemma}

With Lemma~\ref{lem:generalhormander} at hand, Lemma~\ref{lem:densitysmoothing} follows from some relatively standard energy estimates.

\begin{proof}[Proof of Lemma~\ref{lem:densitysmoothing}]
By duality,
$$\|p_\epsilon(t,x_0,\cdot)\|_{L^2} = \sup_{\|f\|_{L^2} \le 1} |P^\epsilon_t f(x_0)| \le \|P_{t}^\epsilon f\|_{L^\infty(B_1)}.$$
Therefore, it is sufficient to show that there exists $p > 0$ that does not depend on $\epsilon$ such that for any $f \in L^2(\R^d)$ we have 
\begin{equation} \label{eq:dualsmooth}
	\|P_1^\epsilon f\|_{L^\infty(B_1)} \leqc \epsilon^{-p}\|f\|_{L^2}.
\end{equation}
 Fix $f \in L^2(\R^d)$ and define $g:[0,\infty) \times \R^d \to \R$ by $g(t,x) = P_t^\epsilon f(x).$ Then, $g$ is a smooth solution to the PDE
	\begin{equation}
		\begin{cases}
			\partial_t g =  \mathcal{L}_\epsilon g, \\
			g|_{t=0} = f,
		\end{cases}
	\end{equation}
	where 
	$$\mathcal{L}_\epsilon = (B-\epsilon Ax) \cdot \grad + \frac{1}{2}\epsilon^3 \sum_{j=1}^d \sigma_j^2 \partial_{x_j}^2 $$
	is the generator of $P_t^\epsilon$. Note that defining $X_{0,\epsilon} = B - \epsilon Ax$ and $X_j = \sigma_j e_j$ we can write 
\begin{equation}\label{eq:horgenerator}
\mathcal{L}_\epsilon = X_{0,\epsilon} + \frac{1}{2} \epsilon^3 \sum_{j=1}^d X_j^2.
\end{equation}
	By Sobolev embedding, to prove \eqref{eq:dualsmooth} it suffices to show that for every $k \in \N$ there is some $q > 0$ and a smooth cutoff $0 \le \chi \le 1$ satisfying $\chi(x) = 1$ for $|x| \le 1$ and $\chi(x) = 0$ for $|x| \ge 2$ such that
	\begin{equation} \label{eq:dualsmooth2}
		\|\chi g(1,\cdot)\|_{H^k} \leqc \epsilon^{-q}\|f\|_{L^2}.
	\end{equation}
	
	We will now prove \eqref{eq:dualsmooth2}. First, note that from $\partial_t g = \mathcal{L}_\epsilon g$ and integration by parts, we have
	$$ \frac{d}{dt}\|g(t,\cdot)\|_{L^2(\R^d)} + \epsilon^3 \sum_{i=1}^d \|X_i g(t,\cdot)\|^2_{L^2(\R^d)} = \int_{\R^d}gX_{0,\epsilon} g \dx = \frac{1}{2} \epsilon \mathrm{Tr}(A)\|g\|_{L^2}^2,$$
	and hence by Gr\"{o}nwall's lemma there holds 
	\begin{equation} \label{eq:baseenergy}
		\|g\|_{L^\infty([0,2];L^2(\R^d))} \leqc \|f\|_{L^2}.
	\end{equation}
	Now for $j \in \N \cup \{0\}$, let $\psi_j:(1/4,2) \to [0,1]$ be a smooth time cutoff with $\psi_j(t) = 1$ for $|t-1| \le 2^{-(j+2)}$ and $\psi_j(t) = 0$ for $|t-1| \ge 2^{-(j+1)}$.  Let $\varphi_j:B_2 \to [0,1]$ be radially symmetric spatial cutoff satisfying $\varphi_j(x) = 1$ for $|x| \le 1+2^{-(j+2)}$ and $\varphi_j(x) = 0$ for $|x| \ge 1+2^{-(j+1)}$. Let $s>0$ be as in Lemma~\ref{lem:generalhormander} applied with $X_{0,\epsilon}$ and $\{X_j\}_{j=1}^d$ as defined in \eqref{eq:horgenerator}. Note that $s$ does not depend on $\epsilon$ and that this application of Lemma~\ref{lem:generalhormander} is justified by our hypoellipticity assumption (2). Let $\chi_j(t,x) = \psi_j(t)\varphi_j(x)$ and set 
	$$g_j = \langle D_{t,x} \rangle^{sj}(\chi_j g) \in C_0^\infty((1/4,2) \times B_2),$$
	where 
	$$\langle D_{t,x} \rangle = \sqrt{1+|\partial_t|^2 + |\grad|^2}$$
	and $|\partial_t|$ and $|\grad|$ are both defined in the natural way as Fourier multipliers. A direct computation using $\partial_t g = \mathcal{L}_\epsilon g$ and $B(x,x) \cdot \grad \varphi_j = 0$ gives, for any $j \in \N \cup \{0\}$, 
	\begin{equation} \label{eq:gjevolution}
		\begin{aligned}
		\partial_t g_j &= X_{0,\epsilon} g_j + \frac{\epsilon^3}{2} \sum_{i=1}^d X_i^2 g_j  \\ 
		 & \quad + \langle D_{t,x} \rangle^{sj}(\varphi_j g \partial_t \psi_j) + \epsilon \langle D_{t,x}
		\rangle^{sj} ([Ax\cdot \grad, \chi_j]g) + \langle D_{t,x} \rangle^{sj} \frac{\epsilon^3}{2} \sum_{i=1}^d [\chi_j,X_i^2]g \\ 
		& \quad + [\langle D_{t,x} \rangle^{sj},X_{0,\epsilon}] \chi_j g.
		\end{aligned}
	\end{equation}
	For convenience, define $\chi_{-1} \equiv 1$. Using \eqref{eq:baseenergy} and that $\chi_{j-1} \equiv 1$ on the support of $\chi_j$ for $j \ge 1$, we obtain the bounds
	\begin{align*}
		\left| \int_0^\infty \int_{\R^d} g_j \langle D_{t,x} \rangle^{sj}([Ax\cdot \grad, \chi_j]g) \dx \dt  \right| &= 	\left| \int_0^\infty \int_{\R^d} g_j \langle D_{t,x} \rangle^{sj}(\chi_{j-1}g Ax\cdot \grad \chi_j)\dx \dt  \right| \\ 
		& \leqc_j \|g_j\|_{L^2(\R^{d+1})}(\mathbf{1}_{j\ge 1}\|g_{j-1}\|_{H^s(\R^{d+1})} + \|f\|_{L^2}),
	\end{align*}
	\begin{align*}
	\left|\int_0^\infty \int_{\R^d} g_j \langle D_{t,x}\rangle^{sj}(\varphi_j g \partial_t \psi_j)\dx \dt \right|\leqc_j \|g_j\|_{L^2(\R^{d+1})}(\mathbf{1}_{j\ge 1}\|g_{j-1}\|_{H^s(\R^{d+1})} + \|f\|_{L^2}),
	\end{align*}
	and 
	\begin{align*}
	&\left| \int_0^\infty \int_{\R^d} g_j \langle D_{t,x}\rangle^{sj}[\chi_j, X_i^2]g \dx \dt\right|  = \left|\int_0^\infty \int_{\R^d} g_j \langle D_{t,x} \rangle^{sj}(\chi_{j-1}g X_i^2 \chi_j + 2 X_i g X_i \chi_j)\dx \dt \right| \\
	& \qquad \le \left|\int_0^\infty \int_{\R^d} g_j \langle D_{t,x} \rangle^{sj}(\chi_{j-1}g X_i^2 \chi_j)\dx \dt\right| + 2 \left|\int_0^\infty \int_{\R^d} X_i g_j \langle D_{t,x} \rangle^{sj}(\chi_{j-1} g X_i \chi_j) \dx \dt \right| \\ 
	& \qquad \qquad + \left|\int_0^\infty \int_{\R^d} g_j \langle D_{t,x} \rangle^{sj}(\chi_{j-1} g X_i^2 \chi_j) \dx \dt \right|\\ 
	& \leqc_j (\|g_j\|_{L^2(\R^{d+1})}+ \|X_i g_j\|_{L^2(\R^{d+1})})(\mathbf{1}_{j \ge 1}\|g_{j-1}\|_{H^s(\R^{d+1})}+\|f\|_{L^2}).
	\end{align*}
	Multiplying \eqref{eq:gjevolution} by $g_j$, integrating over $\R^{d+1}$, and applying the three bounds above then gives 
	\begin{equation} \label{eq:DissipationEnergyEstimate1}
		\epsilon^3 \sum_{i=1}^d \|X_i g_j\|^2_{L^2(\R^{d+1})} \leqc_j \|g_j\|_{L^2(\R^{d+1})}^2 + \mathbf{1}_{j \ge 1}\|g_{j-1}\|_{H^s(\R^{d+1})}^2 + \|f\|_{L^2}^2 + \left|\int_0^\infty \int_{\R^d} g_j [\langle D_{t,x} \rangle^{sj},X_{0,\epsilon}] \chi_j g \dx \dt \right|.
	\end{equation}
	Working on the Fourier side, a relatively standard commutator estimate shows that for any $h \in C_0^\infty((1/4,2)\times B_2)$ there holds 
	\begin{equation}
	\left|\int_0^\infty \int_{\R^d} h [\langle D_{t,x} \rangle^{sj},X_{0,\epsilon}] \chi_j g \dx \dt \right| \leqc \|h\|_{L^2} \|g_j\|_{L^2(\R^{d+1})}.
	\end{equation}
	Hence, we have 
	\begin{equation} \label{eq:DissipationEnergyEstimate2}
		\epsilon^3 \sum_{i=1}^d \|X_i g_j\|^2_{L^2(\R^{d+1})} \leqc_j \|g_j\|_{L^2(\R^{d+1})}^2 + \mathbf{1}_{j \ge 1}\|g_{j-1}\|_{H^s(\R^{d+1})}^2 + \|f\|_{L^2}^2 
	\end{equation}
	Pairing \eqref{eq:gjevolution} with a test function $\psi \in C_0^\infty((1/4,2) \times B_2)$ satisfying $\|\psi\|_{\mathscr{X}} \le 1$, integrating over $\R^{d+1}$, and applying essentially the same estimates as above gives
	\begin{equation} \label{eq:gjdual}
		\|(\partial_t -  X_{0,\epsilon})g_j\|_{\mathscr{X}^*} \leqc \epsilon^3 \sum_{i=1}^d \|X_i g_j\|_{L^2(\R^{d+1})} +\mathbf{1}_{j\ge 1} \|g_{j-1}\|_{H^s(\R^{d+1})} + \|f\|_{L^2} + \|g_j\|_{L^2(\R^{d+1})}.
	\end{equation}
	Putting together \eqref{eq:DissipationEnergyEstimate2} and \eqref{eq:gjdual}, we have proven that for all $j \in \N \cup \{0\}$ there holds
	$$ \|g_j\|_{\mathscr{X}} + \|(\partial_t -  X_{0,\epsilon})g_j\|_{\mathscr{X}^*} \leqc \epsilon^{-3/2}\left(\mathbf{1}_{j\ge 1} \|g_{j-1}\|_{H^s(\R^{d+1})} + \|f\|_{L^2} + \|g_j\|_{L^2(\R^{d+1})}\right).$$
	By Lemma~\ref{lem:generalhormander}, it follows that
	\begin{equation} \label{eq:Horiterate}
		\|g_j\|_{H^s(\R^{d+1})} \leqc \epsilon^{-3/2}\left(\mathbf{1}_{j\ge 1} \|g_{j-1}\|_{H^s(\R^{d+1})} + \|f\|_{L^2} + \|g_j\|_{L^2(\R^{d+1})}\right).
	\end{equation}
	By Young's inequality and \eqref{eq:baseenergy}, there exists $p > 3/2$ with the property for any $\delta > 0$ there is some $C_\delta \ge 1$ such that $$\epsilon^{-3/2}\|g_j\|_{L^2(\R^{d+1})} \le \delta \|g_j\|_{H^s(\R^{d+1})} + C_\delta \epsilon^{-p}\|f\|_{L^2}. $$
	Therefore, the term involving $g_j$ on the right-hand side of \eqref{eq:Horiterate} can be absorbed into the left-hand side, yielding 
	\begin{equation}
		\|g_j\|_{H^s(\R^{d+1})} \leqc \epsilon^{-p}\left(\mathbf{1}_{j\ge 1} \|g_{j-1}\|_{H^s(\R^{d+1})} + \|f\|_{L^2}\right).
	\end{equation}
	Iterating this bound gives \eqref{eq:dualsmooth2}, completing the proof.
\end{proof} 

\subsubsection{Random center-stable manifold}

For any $x_0 \in \S^{d-1}\setminus \Kd$, there exists some $e_j \in \Ec \cap \S^{d-1}$ such that $|x_0 - e_j| \leqc \delta$. In this section, for each of the finitely many $e_j$, we construct a random center-stable manifold for \eqref{eq:scaledSDE} near a line segment of fixed points containing $e_j$. Throughout the section, we fix a single $e_j \in \Ec \cap \S^{d-1}$ to consider and for convenience denote it by $\bar{x}$.

We begin with some notation and basic facts. The linearized dynamics near $\bar{x}$ are determined by the linear map $L_{\bar{x}}:\R^d \to \R^d$ defined by 
$$L_{\bar{x}} (x) = 2 B(\bar{x},x).$$
By the dynamical assumption (1), we have the decomposition
\begin{equation} \label{eq:linearsubs}
	\R^d = E_s \oplus E_u \oplus E_c, 
\end{equation}
where $E_{s}$ is the subspace spanned by the generalized eigenvectors of $L_{\bar{x}}$ corresponding to eigenvalues with negative real part, $E_{u}$ is the subspace spanned by the generalized eigenvectors corresponding to eigenvalues with positive real part, and $E_c = \mathrm{span}(\bar{x})$ is the center subspace. Moreover, $\mathrm{dim}(E_u) \neq 0$. We write $\Pi_{u}$, $\Pi_{s}$, and $\Pi_{c}$ for the associated projections and $$B_{r}^s = \{x \in E_s:|x| < r\}$$ 
for the open ball of radius $r$ centered at the origin in the stable subspace. For $x \in \R^d$ we also define
$$  B_r^s(x) = x + B_r^s$$
and use the corresponding notations with ``s'' replaced by ``u'' or ``ce''  for the closed balls in the unstable and center subspaces, respectively. Note that by  bilinearity, $\alpha \bar{x}$ is a fixed point of $B$ for any $\alpha > 0$ and the stable, unstable, and center subspaces do not change. Moreover, $L_{\alpha \bar{x}} = \alpha L_{\bar{x}}.$ We denote the random solution map associated with \eqref{eq:scaledSDE} by $(t,x) \mapsto \varphi^{\epsilon}_{\omega}(t,x)$. That is, $\varphi_\omega^\epsilon(t,x_0) = x_t^\epsilon(\omega)$, where $x^\epsilon_t$ solves \eqref{eq:scaledSDE} with initial condition $x_0 \in \R^d$. 

We will need a bound for $\varphi^\epsilon_\omega$ that is straightforward consequence of the energy conservation property of $B$.

\begin{lemma} \label{lem:energyconservation}
	Fix $T > 0$ and suppose that $\omega \in \Omega$ is such that 
	$$\sup_{0\le t \le T} \epsilon^{3/2}\sum_{j=1}^d |\sigma_j e_j W_t^{(j)}| \le \epsilon.$$
	There exists a constant $C \ge 1$ that does not depend on $\epsilon$ so that for any $x_0 \in \R^d$ we have 
	\begin{align*} 
		\sup_{0\le t \le T} |\varphi_\omega^\epsilon(t,x_0)| &\le e^{\epsilon C T}\left(|x_0| + C\sqrt{\epsilon}\right) \quad \text{and}  \\ 
		\inf_{0 \le t \le T} |\varphi_\omega^\epsilon(t,x_0)| &\ge e^{-\epsilon C T}|x_0| - C \sqrt{\epsilon}.
		\end{align*}
\end{lemma}

\begin{remark} \label{rem:energycorollary}
	A simple corollary of Lemma~\ref{lem:energyconservation} and some basic properties of Brownian motion is that for any $C \ge 1$ and $\gamma \in (0,1/2)$, there exists $\epsilon_0(C,\gamma) > 0$ such that for all $\epsilon \in (0,\epsilon_0)$ and $x \in \R^d$ with $1/2 \le |x| \le 3/2$ we have 
	$$\P\left(\sup_{0 \le t \le C|\log \epsilon|}(|\varphi_\omega^\epsilon(t,x)|-|x|)\le \gamma\right) \ge 1-\epsilon.$$
\end{remark}

\begin{proof}
	Let
	$$F^\epsilon_t = \epsilon^{3/2}\sum_{j=1}^d \sigma_j e_j W_t^{(j)}$$
	and define $u_t = x^\epsilon_t - F^\epsilon_t.$ Then, 
	$$\frac{du_t}{dt} = B(u_t,u_t) + 2B(u_t,F^\epsilon_t) + B(F^\epsilon_t,F^\epsilon_t) - \epsilon A u_t - \epsilon A F^\epsilon_t. $$
	Using $B(u_t,u_t)\cdot u_t = 0$ and $|F^\epsilon_t| \le \epsilon$, it is easy to show that there exists a constant $C \ge 1$ such that 
	$$-\epsilon^2 - \epsilon C |u_t|^2 \le \frac{d}{dt}|u_t|^2 \le C\epsilon|u_t|^2  + \epsilon^2. $$
	By Gr\"{o}nwall's lemma, it then follows that 
	\begin{equation}
		-\epsilon + e^{-\epsilon C t}|u_0|^2 \le |u_t|^2 \le e^{\epsilon Ct}(|u_0|^2 + \epsilon)
	\end{equation}
	for all $t \in [0,T]$. Taking the square root of this inequality and noting that $u_0 = x_0$ gives 
	\begin{equation}
		e^{-\epsilon C t}|x_0| - \sqrt{\epsilon} \le |u_t| \le e^{\epsilon Ct}(|x_0| + \sqrt{\epsilon}).
	\end{equation}
	The result then follows from the fact that $|u_t| - \epsilon \le |x_t^\epsilon| \le |u_t| + \epsilon$.
\end{proof}

We now use a contraction mapping argument to construct the random center-stable manifold described earlier.

\begin{lemma} \label{lem:centerstable}
	Fix $\bar{x} \in \Ec \cap \S^{d-1}$. If $\delta > 0$ is sufficiently small, then for any $C_0 \ge 1$ there exists $\epsilon_0(\delta, C_0) \in (0,1)$ such that for every $\epsilon \le \epsilon_0$ there is a set $\Omega_\epsilon \subseteq \Omega$ with the following properties:
	\begin{itemize} 
		\item $\P(\Omega_\epsilon) \ge 1-\epsilon$; 
		\item Let $T_\epsilon = C_0|\log \epsilon|$. For every $\omega \in \Omega_\epsilon$ there exists a continuous function $\psi^\epsilon_\omega: B_{\delta}^s + B_{1/2}^{ce}(\bar{x}) \to B_{\delta}^u$ such that for every $x_s \in B_\delta^s$ and $x_c \in B_{1/2}^{ce}(\bar{x})$ there holds 
	\begin{equation} \label{eq:manifold1}
		\sup_{0\le t \le T_\epsilon}|\Pi_s \varphi_\omega^\epsilon(t,x_s + x_c + \psi^\epsilon_\omega(x_s+x_c))| + \sup_{0\le t \le T_\epsilon}|\Pi_c \varphi_\omega^\epsilon(t,x_s + x_c + \psi^\epsilon_\omega(x_s+x_c)) - x_c| \le C \delta
	\end{equation}
	and 
	\begin{equation} \label{eq:manifold2}
	\sup_{0\le t \le T_\epsilon}|\Pi_u \varphi_\omega^\epsilon(t,x_s + x_c + \psi^\epsilon_\omega(x_s+x_c))| \le C \delta^2, 
	\end{equation}
	where $C \ge 1$ is a constant that is independent of $\delta$, $C_0$, and $\epsilon$.
	\end{itemize}
\end{lemma}

\begin{proof}
Throughout, we assume that $\epsilon > 0$ is at least small enough so that for
	$$\Omega_\epsilon = \left\{\omega \in \Omega: \sup_{0 \le t \le T_\epsilon} \epsilon^{3/2}\sum_{j=1}^d |\sigma_j e_j W_t^{(j)}|\le \epsilon\right\}$$
	we have $\P(\Omega_\epsilon) \ge 1- \epsilon$.
	This can be done by the reflection principle and the scaling of $T_\epsilon$ in $\epsilon$. We will construct the function $\psi^\epsilon_\omega$ advertised in the lemma for $\omega \in \Omega_\epsilon$, $\delta$ sufficiently small, and $\epsilon$ perhaps even smaller depending on $\delta$.
	
	Fix $x_s \in B_\delta^s$, $x_c \in B_{1/2}^{ce}(\bar{x})$, and $\omega \in \Omega_\epsilon$. We will define $\psi^\epsilon_\omega(x_s+x_c)$ using a variation of the classical Perron contraction mapping argument. By dynamical assumption (2), there exist $C_1 \ge 1$ and $\lambda > 0$ that do not depend on $x_c$ such that for all $t \ge 0$ and $x \in \R^d$ we have 
	\begin{equation} \label{eq:expdecay}
		|e^{L_{x_c}t}\Pi_s x| \le C_1 e^{-\lambda t} |\Pi_s x| \quad \text{and} \quad 	|e^{-L_{x_c}t}\Pi_u x| \le C_1 e^{-\lambda t}|\Pi_u x|.
	\end{equation}
	 Let $\chi:\R^d \to \R^d$ be a smooth function satisfying $\chi(x) = x$ when $|x| \le 10C_1\delta$ and $\chi(x) = 0$ when $|x| \ge 20 C_1 \delta$.  To simplify notation, we define
	$$F_\epsilon(t) = \epsilon^{3/2} \sum_{j=1}^d \sigma_j e_j W_t^{(j)} \quad \text{and} \quad G_\epsilon(t) = L_{x_c}F_\epsilon(t) + B(F_\epsilon(t), F_\epsilon(t)) - \epsilon A F_\epsilon(t) - \epsilon A x_c,$$
	and then for a continuous function $x:[0,T_\epsilon] \to \R^d$ let 
	$$N_\epsilon(x)(t) = B(\chi(x(t)),\chi(x(t))) + 2B(\chi(x(t)), F_\epsilon(t)) - \epsilon A \chi(x(t)). $$
	Define the mapping $\Phi: C([0,T_\epsilon];\R^d) \to C([0,T_\epsilon];\R^d)$ by 
	\begin{align*} 
		\Phi(x)(t) &= e^{L_{x_c}t}x_s - \int_t^{T_\epsilon}e^{L_{x_c} (t-\tau)}\Pi_u [N_\epsilon(x)(\tau) + G_\epsilon(\tau)]\dee \tau \\ 
		& \quad + \int_0^t e^{L_{x_c} (t-\tau)}\Pi_s [N_\epsilon(x)(\tau) + G_\epsilon(\tau)]\dee \tau + \int_0^t \Pi_c [N_\epsilon(x)(\tau) + G_\epsilon(\tau)]\dee \tau.
	\end{align*}
The function $\Phi$ is chosen so that if the cutoff $\chi$ were removed from the definitions, then a fixed point $x_*:[0,T_\epsilon] \to \R^d$ of $\Phi$ would be such that 
\begin{equation}
	\varphi_\omega^\epsilon(t,x_c + x_*(0)) = x_*(t) + F_\epsilon(t) + x_c.
\end{equation}
Fix $0 < \beta < \lambda$ and let $X$ denote the Banach space of continuous functions $x:[0,T_\epsilon]\to \R^d$ endowed with the norm 
$$\|x\|_{X} = \sup_{0 \le t \le T_\epsilon} |e^{-\beta t}x(t)|.$$
Assuming that $\epsilon < \delta$, it is easy to see that there exists a constant $C_2 > 0$ such that for any $x, \tilde{x} \in X$ we have 
$$|N_\epsilon(x)(t) - N_\epsilon(\tilde{x})(t)| \le C_2 \delta e^{\beta t}\|x - \tilde{x}\|_{X}.$$
Therefore, using \eqref{eq:expdecay} to bound the term that involves $\Pi_u$, we obtain
\begin{align*}
	|\Phi(x)(t) - \Phi(\tilde{x})(t)| &\le C_1C_2 \delta \|x - \tilde{x}\|_{X}\left( \int_t^{T_\epsilon} e^{-\lambda(\tau-t)}e^{\beta \tau} \dee \tau + 2  \int_0^t e^{\beta \tau} \dee \tau\right) \\ 
	& \le C_1 C_2 \delta \|x - \tilde{x}\|_{X} \left(\frac{e^{\beta t}}{\lambda - \beta} + \frac{2 e^{\beta t}}{\beta}\right).
\end{align*}
It follows that $\Phi:X \to X$ is a contraction provided that $\delta$ is chosen such that 
$$ \delta < \left(C_1 C_2 \left(\frac{1}{\lambda - \beta} + \frac{2}{\beta}\right)\right)^{-1}.$$
By the Banach fixed point theorem, there exists a unique $x_* \in X$ satisfying $\Phi(x_*)(t) = x_*(t)$ for all $t \in [0,T_\epsilon]$. We set 
\begin{equation} \label{eq:graphdef}
	\psi^\epsilon_\omega(x_c+x_s) = -\int_0^{T_\epsilon}e^{L_{x_c}(t-\tau)}\Pi_u[N_\epsilon(x_*)(\tau)+G_\epsilon(\tau)]\dee \tau.
\end{equation}
An estimate using \eqref{eq:expdecay} similar to the ones above shows that $|\psi^\epsilon_\omega(x_c + x_s)| \leqc \delta^2$ if $\epsilon < \delta^2$, so we indeed have $\psi^\epsilon_\omega: B_\delta^s + B_{1/2}^{ce}(\bar{x}) \to B_\delta^u$ if $\delta$ is chosen sufficiently small. 

We now verify that $\psi^\epsilon_\omega$ as defined above satisfies \eqref{eq:manifold1} and \eqref{eq:manifold2}. Let $0 \le T \le T_\epsilon$ be the maximal time such that 
\begin{equation}
	\chi(x_*(t)) = x_*(t) \quad \forall 0 \le t \le T.
\end{equation}
From the pathwise uniqueness of solutions to \eqref{eq:scaledSDE} and the fact that 
\begin{equation} \label{eq:fixedptODE}
	\frac{d}{dt}x_*(t) = L_{x_c} x_*(t) + B(\chi(x_*(t)), \chi(x_*(t))) + L_{x_c} F_\epsilon(t) + 2B(\chi(x_*(t)),F_\epsilon(t)) - \epsilon A \chi(x_*(t))- \epsilon Ax_c,
\end{equation}
we have 
\begin{equation} \label{eq:varphibootstrap}
	\varphi_\omega^\epsilon(t,x_s+x_c+\psi^\epsilon_\omega(x_s+x_c)) = x_*(t) + x_c + F_\epsilon(t), \quad t \in [0,T].
\end{equation}
We would like to show that $T = T_\epsilon$. From $\Phi(x_*) = x_*$ and estimates similar to those in the previous paragraph, it is immediate that there exists $C_3 > 0$ such that for $\delta$ sufficiently small (independent of $\epsilon$) we have
\begin{equation} \label{eq:fixedptbounds}
	\sup_{0\le t \le T_\epsilon}|\Pi_s x_*(t)| \le 2C_1 \delta  \quad \text{and} \quad \sup_{0 \le t \le T_\epsilon}|\Pi_u x_*(t)| \le C_3 \delta^2.
\end{equation} 
In particular, if $\delta > 0$ is sufficiently small, then $T > 0$ and in order to show $T = T_\epsilon$ it is enough to prove that 
\begin{equation} \label{eq:y*goal}
	\sup_{0\le t \le T}|\Pi_c x_*(t)| \le 7 C_1 \delta.
\end{equation}
 Now, by \eqref{eq:varphibootstrap}, Lemma~\ref{lem:energyconservation}, and the fact that $|F_\epsilon(t)| \le \epsilon$ for all $t \in [0,T_\epsilon]$, there exists some $C_4 > 0$ that does not depend on $\epsilon$ or $\delta$ such that for all $t \in [0,T]$ we have
\begin{equation} \label{eq:y*bound1}
	e^{-\epsilon C_4 T}(|x_c| - |x_s| -  |\psi^\epsilon_\omega(x_s+x_c)|) - C_4 \sqrt{\epsilon} \le |x_*(t) + x_c| \le e^{\epsilon C_4 T}(|x_c| + |x_s| + |\psi^\epsilon_\omega(x_s+x_c)| + C_4 \sqrt{\epsilon}).
\end{equation}
We may assume that $\epsilon > 0$ is small enough so that $e^{\epsilon C_4 T_\epsilon} \le 1+\sqrt{\epsilon}$ and $e^{-\epsilon C_4 T_\epsilon} \ge 1-\sqrt{\epsilon}$. Putting this bound into \eqref{eq:y*bound1} and recalling also that $|x_s| + |\psi_\omega(x_s+x_c)| \le 2 \delta$, we see that there exists $C_5 > 0$ such that 
\begin{equation} \label{eq:y*bound2}
	|x_c| - (2\delta + C_5\sqrt{\epsilon}) \le |x_*(t) + x_c| \le |x_c| + (2\delta + C_5\sqrt{\epsilon}), \quad t \in [0,T].
\end{equation}
It follows from \eqref{eq:fixedptbounds} and \eqref{eq:y*bound2} that for $\delta$ sufficiently small we have 
\begin{equation} \label{eq:y*bound3}
\left||\Pi_c x_*(t)+x_c| - |x_c|\right| \le 2\delta + 2C_1\delta + C_3 \delta^2 + C_5\sqrt{\epsilon}\le 5C_1 \delta + C_5\sqrt{\epsilon}, \quad t \in [0,T].
\end{equation}
Since the center subspace is one-dimensional, $\Pi_c x_*(0) = 0$, and $x_*$ is continuous, \eqref{eq:y*bound3} actually implies that $|\Pi_c x_*(t)| \le 5C_1 \delta + C_5 \sqrt{\epsilon}$ for all $t \in [0,T]$. This gives \eqref{eq:y*goal} and hence also $T= T_\epsilon$ provided that $\epsilon \ll \delta^2$. The bounds \eqref{eq:manifold1} and \eqref{eq:manifold2} then follow immediately from \eqref{eq:varphibootstrap} and \eqref{eq:fixedptbounds}.
\end{proof}

\subsubsection{Concluding the proof of the exit time estimates}

The last ingredient we need before completing the proof of Lemma~\ref{lem:scaledexit} is a lemma that describes the growth of perturbations in $E_u$ of random trajectories that remain near $\Ec$.

\begin{lemma} \label{lem:unstablepert}
	Fix $\bar{x} \in \Ec \cap \S^{d-1}$ and let $E_s$, $E_u$, and $E_c$ be the associated linear subspaces defined in \eqref{eq:linearsubs}. Let $x_c \in B_{1/2}^{ce}(\bar{x})$ and suppose that $x_0 \in \R^d$, $\omega \in \Omega$, $\delta \in (0,1)$ and $C_0,T \ge 1$ are such that 
	\begin{equation}
		\sup_{0 \le t \le T}|\varphi_\omega^\epsilon(t, x_0) - x_c| \le C_0\delta. \\
	\end{equation}
For $v \in E_u$, define 
	$$t_*(v) = \sup\{0 \le t \le T: |\Pi_u (\varphi_\omega^\epsilon(s,x_0+v) - \varphi_\omega^\epsilon(s,x_0))| \le 2\delta \quad \forall 0 \le s \le t\}.$$
	There exist $c > 0$, $\lambda > 0$, and $\delta_0,\epsilon_0 \in (0,1)$ so that if $\delta \in  (0,\delta_0)$, $\epsilon \in (0,\epsilon_0)$, and $v \in B_{2\delta}^u$, then we have 
	$$ |\Pi_u(\varphi_\omega^\epsilon(t,x_0+v) - \varphi_\omega^\epsilon(t,x_0))| \ge c|v|e^{\lambda t} \quad \forall 0\le t \le t_*(v). $$
	The constants $c$, $\lambda$, $\delta_0$, and $\epsilon_0$ 
	depend only on $C_0$ and properties of the linearized operator $L_{\bar{x}}$.
\end{lemma}

\begin{proof}
	In this proof, $c$ and $C$ denote generic positive constants that for $\epsilon, \delta > 0$ sufficiently small depend only on $C_0$ and $L_{\bar{x}}$. 
	There exists a real change of variables matrix $P$ that leaves $E_u$ and $E_s$ invariant (and acts as the identity on $E_c$) and is such that, defining $R_{x_c} = P L_{x_c} P^{-1}$, for every $x \in \R^d$ we have 
	\begin{align}
		 R_{x_c} \Pi_u x \cdot \Pi_u x &\ge c |\Pi_u x|^2,  \label{eq:R1}\\
		  R_{x_c} \Pi_s x \cdot \Pi_s x &\le -c |\Pi_s x|^2, \label{eq:R2}\\ 
		 R_{x_c} \Pi_c x &= 0\label{eq:R3}, \quad \text{and} \\ 
		 |P x| &\le C|x|. \label{eq:R4}
	\end{align}
		Let $$h(t) = P\left(\varphi_\omega^\epsilon(t,x_0 + v) - \varphi_\omega^\epsilon(t,x_0)\right).$$
		Then, define
		$$\Phi(t) = |\Pi_u h(t)|^2 - |\Pi_s h(t)|^2 - |\Pi_c h(t)|^2$$	
		and let $t_*' \ge 0$ denote the maximal time such that $\Phi(t)  > 0$ for all $0 \le t \le t_*'$. Note that $t_*' > 0$ by the continuity of $\Phi$ and fact that $\Phi(0) = |P v|^2 > 0$. Moreover, from the definitions of $\Phi$ and $t_*$ it follows easily that 
		\begin{equation}\label{eq:hupperbound}
			|h(t)| \le C\delta 
		\end{equation}
		for all $t \in [0,t_* \wedge t_*'].$ For simplicity of notation, define $g(x) = B(x,x)$ and $\bar{A} = PAP^{-1}$. A direct computation then shows that $h$ solves 
	\begin{equation}
		\begin{cases}
			\frac{dh}{dt} = R_{x_c} h(t) - \epsilon \bar{A} h(t) + P\left(g(P^{-1}h(t)+\varphi_\omega^\epsilon(t,x_0) - x_c) - g(\varphi_\omega^\epsilon(t,x_0) - x_c)\right), \\ 
			h(0) = Pv.
		\end{cases}
	\end{equation}
	Note that since $g$ is smooth and vanishes at zero, from $|\varphi_\omega^\epsilon(t,x_0) - x_c| \le C_0 \delta$ and \eqref{eq:hupperbound} we have 
	\begin{equation} \label{eq:nonlinearpert}
		\left|P\left(g(P^{-1}h(t)+\varphi_\omega^\epsilon(t,x_0) - x_c) - g(\varphi_\omega^\epsilon(t,x_0) - x_c)\right)\right|\le C \delta |h(t)| \quad \forall 0\le t\le t_*\wedge t_*'.
	\end{equation}
	Using \eqref{eq:R1}, \eqref{eq:R2}, \eqref{eq:R3}, \eqref{eq:R4}, and \eqref{eq:nonlinearpert} we obtain, for $0 \le t \le t_* \wedge t_*'$, the differential inequalities
	\begin{align}
		\frac{d}{dt}|\Pi_u h(t)|^2 &\ge c|\Pi_u h(t)|^2 - (C\epsilon + C\delta) |\Pi_u h(t)| |h(t)| \\ 
		& \ge (c - C\delta - C\epsilon)|\Pi_u h(t)|^2 - (C\epsilon + C\delta)(|\Pi_s h(t)|^2 + |\Pi_c h(t)|^2),
	\end{align}
		\begin{align}
		\frac{d}{dt}|\Pi_s h(t)|^2 &\le -c|\Pi_s h(t)|^2 + (C\epsilon + C\delta) |\Pi_s h(t)| |h(t)| \\ 
		& \le -(c - C\delta - C\epsilon)|\Pi_s h(t)|^2 + (C\epsilon + C\delta)(|\Pi_u h(t)|^2 + |\Pi_c h(t)|^2),
	\end{align}
	and 
	\begin{align}
		\frac{d}{dt}|\Pi_c h(t)|^2 \le (C\epsilon + C\delta) |h(t)|^2. \\ 
	\end{align}
Combining these inequalities and taking $\epsilon$ and $\delta$ sufficiently small gives 
$$\frac{d}{dt}\Phi(t) \ge c(|\Pi_u h(t)|^2 + |\Pi_s h(t)|^2) - (C\delta + C\epsilon)|\Pi_c h(t)| \ge c\Phi(t),$$
where in the second inequality we assumed that $C\delta + C\epsilon \le c$. The previous bound holds for all $t \le t_* \wedge t_*'$. Therefore, $t_*' \ge t_*$ by continuity of $\Phi$ and we obtain
$$|\Pi_u h(t)|^2 \ge \Phi(t) \ge e^{c t}\Phi(0) = e^{ct}|P v|^2$$
for all $t \le t_*$, which completes the proof.
\end{proof}

We are now ready to complete the proof of Lemma~\ref{lem:scaledexit} by combining the hypoelliptic smoothing described by Lemma~\ref{lem:densitysmoothing} with the other results of this section.

\begin{proof}[Proof of Lemma~\ref{lem:scaledexit}]
		
	Fix $\bar{x} \in \Ec \cap \S^{d-1}$ and let $E_s$, $E_u$, and $E_c$ be the associated linear subspaces defined earlier. For $0 < \delta \ll 1$, let $$V_\delta = \{x \in B_\delta^s + B_\delta^u + E_c: 3/4 < |x| < 5/4 \text{ and } \Pi_c x \cdot \bar{x} \ge 0\}.$$ 
	Note that $V_\delta$ is simply the connected component of the set $(B_\delta^s + B_\delta^u + E_c) \cap \{3/4 < |x| < 5/4\}$ that contains $\bar{x}$. We define the stopping time 
	$$\tilde{\tau}_\delta^\epsilon(x_0,\omega) = \inf\{t \ge 0: x_t^\epsilon(\omega) \not \in V_\delta\},$$
	where $x_t^\epsilon$ solves \eqref{eq:scaledSDE} with initial condition $x_0 \in \R^d$. We will prove that there exists $C_1 \ge 1$ so that for all $\delta > 0$ sufficiently small and $x_0 \in \S^{d-1} \cap V_\delta $ we have
	\begin{equation} \label{eq:modifiedstopping}
	\P\left(\tilde{\tau}_\delta^\epsilon(x_0,\omega) \le C_1|\log \epsilon|+1\right) \ge \frac{1}{2}
	\end{equation}
	provided that $\epsilon$ is taken small enough depending on $\delta$. Once \eqref{eq:modifiedstopping} is established, Lemma~\ref{lem:scaledexit} will follow easily. To prove \eqref{eq:modifiedstopping}, let us first observe that 
	$$\{\omega \in \Omega: x_t^\epsilon(\omega) \not \in 
	V_\delta \text{ for some }1 \le t \le 1+ C_1|\log \epsilon|\} \subseteq  \{\omega \in \Omega: \tilde{\tau}_\delta^\epsilon(x_0,\omega) \le C_1|\log \epsilon|+1\}.$$
	Therefore, it follows by the Markov property and Fubini's theorem that 
	\begin{align}
		\P(\tilde{\tau}^\epsilon_\delta(x_0,\omega) \le C_1|\log \epsilon|+1) &\ge \int_{\R^d} p_\epsilon(1,x_0,x) \P(\tilde{\tau}^\epsilon_\delta(x,\omega)\le C_1|\log \epsilon|)\dx \\ 
		& = \int_{\Omega} \left(\int_{\R^d}\mathbf{1}_{\{\tilde{\tau}_\delta^\epsilon(x,\omega)\le C_1|\log\epsilon|\}} p_\epsilon(1,x_0,x)\dee x\right) \P(\dee \omega). \label{eq:MarkovFubini}
	\end{align}
We assume throughout that $\delta$ is small enough so that Lemma~\ref{lem:centerstable} applies. Let $\Omega_\epsilon$ be the set obtained from applying Lemma~\ref{lem:centerstable} with $T_\epsilon = C_1 |\log \epsilon|$ and let $C \ge 1$ be the constant so that \eqref{eq:manifold1} and \eqref{eq:manifold2} both hold. 
	From \eqref{eq:MarkovFubini}, it holds trivially that
	 $$
	 \P(\tilde{\tau}_\delta^\epsilon(x_0,\omega) \le C_1|\log \epsilon|+1) \ge \int_{\Omega_\epsilon} \left(\int_{\R^d}\mathbf{1}_{\{\tilde{\tau}_\delta^\epsilon(x,\omega)\le C_1|\log \epsilon|\}} p_\epsilon(1,x_0,x)\dee x\right) \P(\dee \omega).
	 $$
	 Since $\P(\Omega_\epsilon) \ge 1-\epsilon$ for $\epsilon$ small, to complete the proof of \eqref{eq:modifiedstopping} it is enough to show that 
	 \begin{equation}\label{eq:Aepsgoal}
	 	\int_{\R^d}\mathbf{1}_{\{\tilde{\tau}_\delta^\epsilon(x,\omega)\le C_1|\log \epsilon|\}} p_\epsilon(1,x_0,x)\dee x \ge \frac{3}{4} \quad \forall \omega \in \Omega_\epsilon
	 \end{equation}
	 provided that $\epsilon$ and $\delta$ are sufficiently small, with $\delta$ chosen independently of $\epsilon$. Fix $\omega \in \Omega_\epsilon$. Let $\psi^\epsilon_\omega: B_\delta^s + B_{1/2}^{ce}(\bar{x}) \to B_\delta^u$ be the function guaranteed by Lemma~\ref{lem:centerstable} and for $r > 0$ define 
	 $$\mathcal{U}_r = \{x+\psi^\epsilon_\omega(x)+v:  x \in B_\delta^s + B_{1/2}^{ce}(\bar{x}) \text{ and } v \in B_r^u\}.$$
	 By Fubini's theorem, we have 
	 \begin{equation}\label{eq:volbound}
	 	|\mathcal{U}_r| \le \alpha r^{\beta}
	 \end{equation}
	 for some positive constants $\alpha$ and $\beta$ that depend on $d$ and the dimension of $E_u$. Splitting the integral in \eqref{eq:Aepsgoal} between the sets $\R^d \setminus V_\delta$, $V_\delta \cap \mathcal{U}_r$, and $V_\delta \setminus \mathcal{U}_r$ and using that the characteristic function is trivially equal to one for $x \in \R^d \setminus V_\delta$ gives
	 \begin{equation}\label{eq:densitysplit}
	 	\begin{aligned} \int_{\R^d}\mathbf{1}_{\{\tilde{\tau}_\delta^\epsilon(x,\omega)\le C_1|\log\epsilon|\}} p_\epsilon(1,x_0,x)\dee x &\ge 1 - \int_{V_\delta \cap \mathcal{U}_r} p_\epsilon(1,x_0,x)\dx \\
	 		& \quad - \int_{V_\delta \setminus \mathcal{U}_r} (1-\mathbf{1}_{\{\tilde{\tau}_\delta^\epsilon(x,\omega)\le C_1|\log\epsilon|\}}) p_\epsilon(1,x_0,x)\dx.
	 		\end{aligned} 
	 	\end{equation}
For the first integral on the right-hand side of \eqref{eq:densitysplit}, taking $r = \epsilon^m$ for $m \ge 1$ sufficiently large it follows by H\"{o}lder's inequality and Lemma~\ref{lem:densitysmoothing} that 
\begin{equation}  \label{eq:Urbound}
	\left|\int_{V_\delta \cap \mathcal{U}_r} p_\epsilon(1,x_0,x)\dx \right| \le \epsilon
\end{equation}
for all $\epsilon$ sufficiently small. For the second integral on the right-hand side of \eqref{eq:densitysplit}, we claim that it is zero provided that $C_1$ is large enough. In what follows, we denote $\Pi_{sc} = \Pi_s + \Pi_c$ for convenience. Since $\Pi_{sc} V_\delta \subseteq B_\delta^s + B_{1/2}^c(\bar{x})$ for $\delta$ small, $\psi^\epsilon_\omega(\Pi_{sc}x)$ is well defined for each $x \in V_\delta$. For $x \in V_\delta$, let $\Psi_\omega^\epsilon(x) = \Pi_{sc}x + \psi_\omega^\epsilon(\Pi_{sc}x)$ and $v(x) = \Pi_u x - \psi_\omega^\epsilon(\Pi_{sc}x) \in E_u$. Then, for $x \in V_\delta\setminus \mathcal{U}_r$, we may write $\varphi_\omega^\epsilon(t,x)$ as
\begin{align*}\varphi_\omega^\epsilon(t,x) &= [\varphi_\omega^\epsilon(t,\Psi_\omega^\epsilon(x) + v(x)) - \varphi_\omega^\epsilon(t,\Psi_\omega^\epsilon(x))] +  \varphi_\omega^\epsilon(t,\Psi_\omega^\epsilon(x)).
\end{align*}
Let $$t_* = \inf\{0\le t \le C_1 |\log \epsilon|: \left|\Pi_u(\varphi_\omega^\epsilon(t,\Psi_\omega^\epsilon(x) + v(x)) - \varphi_\omega^\epsilon(t,\Psi_\omega^\epsilon(x)))\right| \ge 2\delta\}.$$ 
Since 
$$\sup_{0 \le t \le C_1 |\log \epsilon|}|\varphi_\omega^\epsilon(t,\Psi_\omega^\epsilon(x)) - \Pi_c x| \le 2C\delta$$
by Lemma~\ref{lem:centerstable}, it follows from Lemma~\ref{lem:unstablepert} applied with $\Pi_c x \in B_{1/2}^{ce}(\bar{x})$ that there exist $c, \lambda > 0$ that do not depend on $x$ or $C_1$ such that for $t \le t_*$ and $\delta > 0$ sufficiently small depending on $C$ we have 
$$\left|\Pi_u(\varphi_\omega^\epsilon(t,\Psi_\omega^\epsilon(x) + v(x)) - \varphi_\omega^\epsilon(t,\Psi_\omega^\epsilon(x)))\right| \ge c|v(x)| e^{\lambda t}. $$
Recalling \eqref{eq:manifold2} and that $|v(x)| \ge \epsilon^m$ by the definition of $\mathcal{U}_r$, we obtain
$$|\Pi_u \varphi_\omega^\epsilon(t,x)| \ge c\epsilon^m e^{\lambda t} - C\delta^2$$
for all $t \le t_*$. It follows that if 
\begin{equation} \label{eq:C1}
	C_1 = \frac{1}{\lambda}\left(\max(0,\log(3/c)) + m\right) \quad \text{and} \quad \delta \le \frac{1}{2C},
\end{equation}
then $t_* \le C_1|\log \epsilon|$.  That is, there exists $t_0 \le C_1|\log\epsilon|$ such that $|\Pi_u \varphi_\omega^\epsilon(t_0,x)| \ge 2\delta$, implying that $\varphi_\omega^\epsilon(t_0,x) \not\in V_\delta$.  With this choice of $C_1$, we thus have 
\begin{equation} \label{eq:characteristic1}
	\int_{V_\delta \setminus \mathcal{U}_r} (1-\mathbf{1}_{\{\tilde{\tau}_\delta^\epsilon(x,\omega)\le C_1|\log\epsilon|\}}) p_\epsilon(1,x_0,x)\dx = 0. 
\end{equation}
Putting \eqref{eq:Urbound} and \eqref{eq:characteristic1} into \eqref{eq:densitysplit} proves \eqref{eq:Aepsgoal}, and hence completes the proof of \eqref{eq:modifiedstopping}.

\begin{remark}
	Before concluding the proof, we remark for the careful reader that in the arguments that led to \eqref{eq:modifiedstopping}, the constants $\epsilon$, $\delta$, and $C_1$ have been tuned in a consistent way. The constant $\delta > 0$ was first chosen small enough so that Lemma~\ref{lem:centerstable} applies, and then perhaps smaller depending on $C$ so that \eqref{eq:C1} holds and Lemma~\ref{lem:unstablepert} is valid. Picking $C_1 \ge 1$ in accordance with \eqref{eq:C1} is always possible because the constants $c$ and $\lambda$ obtained from our application of Lemma~\ref{lem:unstablepert} did not depend on the choice of $C_1$ that was to be determined. With $\delta$ and $C_1$ chosen, $\epsilon(C_1,\delta)$ was then picked small enough as required by Lemma~\ref{lem:centerstable} and \eqref{eq:Urbound}.
\end{remark}

Lastly, we argue that \eqref{eq:modifiedstopping} implies Lemma~\ref{lem:scaledexit}. Let $V_\delta^i$ be defined in the same way as $V_\delta$ but with $\bar{x}$ replaced by $e_i$, and let $\tilde{\tau}_{\delta,i}^\epsilon$ have the same definition as $\tilde{\tau}_\delta^\epsilon$ but with $V_\delta$ replaced by $V_\delta^i$. While \eqref{eq:modifiedstopping} was proven for a fixed $\bar{x} \in \Ec\cap \S^{d-1}$, it clearly implies that there exists $C_1 \ge 1$ such that for all $\delta', \epsilon > 0$ sufficiently small there holds 
\begin{equation}\label{eq:modifiedstopping2}
	\P(\tilde{\tau}_{\delta',i}^\epsilon(x,\omega) \le C_1|\log \epsilon| + 1) \ge \frac{1}{2}
\end{equation}
for any $1 \le i \le d$ and $x \in V_{\delta'}^i$. 
Define 
$$\tau^\epsilon(x,\omega) = \inf\{t \ge 0: ||\varphi_\epsilon(t,x,\omega)| - 1| \ge 1/4\}.$$
 We may choose $0 \le \delta \ll \delta' \ll 1$ so that $\S^{d-1}\setminus \Kd \subseteq \cup_{i=1}^dV_{\delta'}^i$ and for every $x \in \S^{d-1}\setminus \Kd$ there exists some $1 \le i \le d$ for which
\begin{equation} \label{eq:stopping1}
	\P(\tau_\delta^\epsilon(x,\omega) \le C_1|\log \epsilon|+1) \ge  \P(\tilde{\tau}_{\delta',i}^\epsilon(x,\omega) \le C_1|\log \epsilon|+1) - \P(\tau^\epsilon(x,\omega)<C_1|\log \epsilon|+1). 
\end{equation}
The inclusion $\S^{d-1}\setminus \Kd \subseteq \cup_{i=1}^dV_{\delta'}^i$ and \eqref{eq:modifiedstopping2} imply that
 $$\P(\tilde{\tau}_{\delta',i}^\epsilon(x,\omega) \le C_1|\log \epsilon|+1) \ge \frac{1}{2},$$
 and so the desired result of Lemma~\ref{lem:scaledexit} follows by using Remark~\ref{rem:energycorollary} to bound the second term on the right-hand side of \eqref{eq:stopping1} for $\epsilon$ sufficiently small.
\end{proof}


\subsection{Concluding the proof of Proposition~\ref{prop:dynCondImpliesExist1}} \label{sec:finishproof}

In this section we use Lemma~\ref{lem:scaledexit} to complete the proof of Proposition~\ref{prop:dynCondImpliesExist1}. The one additional ingredient that we need is a lower bound for the time average $\E\int_0^1 Ax_t^\epsilon \cdot x_t^\epsilon \dt$ when the initial condition $x_0 \in \S^{d-1}$ is not too close to any of the equilibria of $B$. To this end, we first have a statement that follows easily from assumption (3) and a compactness argument. 
\begin{lemma} \label{lem:compactness}
	For $X_0 \in \R^d$, let $X_t$ denote the solution of the ODE 
	\begin{equation} \label{eq:conservativeODE}
		\begin{cases}
			\frac{d}{dt} X_t = B(X_t,X_t), \\
			X_t|_{t=0} = X_0.
		\end{cases}
	\end{equation}
	Fix $\delta > 0$. Under the dynamical assumption (3), there exists $J_\delta \in \N$ and $c_\delta > 0$ such that for all $X_0 \in \Kd$ there is some $0 \le j \le J_\delta$ such that 
	\begin{equation}
		\left|\Pi_{\Kc}^\perp\frac{d^j}{dt^j} X_t|_{t=0} \right| \ge c_\delta.
	\end{equation}
\end{lemma}

\begin{proof}
	
By the dynamical assumption (3), for every $X_0 \in \Kd$ there is some $j(X_0) \in \N \cup \{0\}$ such that
	\begin{equation} \label{eq:nonzeroderivative}
	\Pi_{\Kc}^\perp\frac{d^j}{dt^j} X_t|_{t=0}  \neq 0.
	\end{equation}
Note now that
$$\frac{d^j}{dt^j}X_t|_{t=0} = P_j(X_0)$$
	for some homogeneous polynomial $P_j$ of degree $j+1$. It follows then by \eqref{eq:nonzeroderivative} that for each $X_0 \in \Kd$ there is some open set $\mathcal{U}_{X_0}$ containing $X_0$ such that $$\Pi_{\Kc}^\perp P_{j(X_0)}(x) \neq 0$$
	for every $x \in \mathcal{U}_{X_0}$. Since $\Kd$ is compact, there exists $m \in \N$ and $\{X_0^i\}_{i=1}^m \subseteq \Kd$ such that $\Kd \subseteq \cup_{i=1}^m \mathcal{U}_{X_0^j}$. Let $J_\delta = \max_{1\le i \le m}j(X_0^i)$. Then, for every $x \in \Kd$ we have $P_j(x) \neq 0$ for some $j \le J_\delta$. The result then follows from the extreme value theorem applied with the continuous function 
	$$x \mapsto \sum_{j=1}^{J_\delta}|P_j(x)|.$$
\end{proof}

Lemma~\ref{lem:compactness} and the approximation arguments in \cite[Section 3]{BedrossianLiss22} give us a lower bound for $\E\int_0^1 Ax_t^\epsilon \cdot x_t^\epsilon \dt$ when $x_0 \in \Kd$. We omit the details of the proof for the sake of brevity.

\begin{lemma} \label{lem:conservativediss}
	Fix $\delta > 0$. Suppose that there exists $c > 0$ and $J \in \N$ such that for every $X_0 \in \Kd$ the solution of \eqref{eq:conservativeODE} satisfies
		\begin{equation}
		\left|\Pi_{\Kc}^\perp\frac{d^j}{dt^j} X_t|_{t=0} \right| \ge c
	\end{equation}
	for some $0 \le j \le J$. Then, there exists $c_\delta > 0$ depending on $c$ and $J$ such that for all $\epsilon$ sufficiently small and $x_0 \in \Kd$ we have 
\begin{equation} \label{eq:cdelta}
\E\int_0^1 Ax^\epsilon_t \cdot x^\epsilon_t \dt \ge c_\delta.
\end{equation}
\end{lemma}

We can now finally conclude the proof of Proposition~\ref{prop:dynCondImpliesExist1}.

\begin{proof}[Proof of Proposition~\ref{prop:dynCondImpliesExist1}]
	
Let $C,c,\delta > 0$ be as in the statement of Lemma~\ref{lem:scaledexit} and let $c_\delta >0$ be as in Lemma~\ref{lem:conservativediss}, which we may apply here by Lemma~\ref{lem:compactness} and dynamical assumption (3). As described earlier in Section~\ref{sec:rescaled}, it is sufficient to show that there are constants $c_0, C_0 > 0$ such that for all $\epsilon>0$ sufficiently small and $x_0 \in \S^{d-1}$ we have 
	\begin{equation} \label{eq:finishproof}
		\E \int_0^{C_0 |\log \epsilon|}Ax^\epsilon_t \cdot x^\epsilon_t \dt \ge c_0.
	\end{equation}
We claim that \eqref{eq:finishproof} holds with $C_0 = C+1$ and $c_0 = c c_\delta$, provided that $\epsilon$ is small enough so that both Lemmas~\ref{lem:scaledexit} and~\ref{lem:conservativediss} apply. If $x_0 \in \Kd \cap \S^{d-1}$, then the bound is immediate from Lemma~\ref{lem:conservativediss}. Suppose then that $x_0 \in \S^{d-1} \setminus \Kd$. For simplicity of notation, let $\tau(\omega) = \inf\{t \ge 0:x_t^\epsilon(\omega) \in \Kd\}$ and define $\mathcal{A}$ to be the event that $\tau \le C|\log \epsilon|$. Let $\mathcal{F}_\tau$ be the $\sigma$-algebra of events determined prior to the stopping time $\tau$. Define also $D(x) = Ax \cdot x$. It follows then by the strong Markov property, Lemma~\ref{lem:scaledexit},  and Lemma~\ref{lem:conservativediss} that 
\begin{align*}
	\E \int_0^{(C+1)|\log \epsilon|} A x_t^\epsilon \cdot x_t^\epsilon \dt &\ge \int_{\mathcal{A}} \int_0^{C|\log \epsilon| + 1}A x_t^\epsilon \cdot x_t^\epsilon \dt \P(\dee \omega) \\ 
	& \ge \int_{\mathcal{A}} \int_0^1 D(x_{\tau+t}^\epsilon) \dt \P(\dee \omega) \\ 
	& = \int_0^1 \int_{\mathcal{A}} \E(D(x^\epsilon_{\tau+t})|\mathcal{F}_\tau)\P(\dee \omega) \dt \\ 
	& = \int_{\mathcal{A}} \left(\int_0^1 P_t^\epsilon D(x^\epsilon_\tau) \dt \right)\P(\dee \omega) \\ 
	& \ge \P(\mathcal{A}) c_\delta \ge c c_\delta. 
\end{align*}
\end{proof}

\section{Generic nonlinearities}
\label{sec:genericConditions1}

\newcommand{\Disc}{\operatorname{Disc}}
\newcommand{\Hc}{\mathcal{H}}
\newcommand{\sym}{{\rm sym}}

To recap, the goal of this section is to prove Proposition \ref{prop:genericconditions}: we seek to show that for an open, dense and full Lebesgue-measure subset of $b \in \Bc$ that the vector field $B_b$ has the following properties: given $J < \frac{2d}{3}$ and $\set{\sigma_i}_{1 \leq i \leq d}$ for which $|\set{i : \sigma_i \neq 0}| \geq 2$, one has that
\begin{itemize}
	\item[(1)] (Hyperbolicity) The linearization $D_x B$ at each equilibrium of $\Ec$ admits a single eigenvalue along the imaginary axis; 
\item[(2)] (Hypoellipticity) The collection of vector fields $\{B_b + \eps Mx, \sigma_1 e_1, \ldots, \sigma_d e_d\}$  satisfies the parabolic H\"ormander condition for all $\eps \in [0,1]$ and for any fixed $d \times d$ matrix $M$. 
	\item[(3)] (Dynamics on $\Kc$) For any solution $x(t)$ to the deterministic problem $\dot x = B_b(x,x)$, it holds that if $x(t) \in \Kc \setminus \Ec$, then
	\[\frac{d^j}{dt^j} x(t) \notin \mathcal{\Kc} \]
	for some $j \geq 1$. 
\end{itemize}

The plan this section is as follows. In Section \ref{subsec:constraintClass3} we provide an explicit description of the constraint class from Definition \ref{def:CCI} and give some comments for how we will practically work with it. Section \ref{subsec:genercHyperbolicity3} establishes generic hyperbolicity of the equilibria along $\Ec$ as in (1) above, and Section \ref{subsec:genericHypoellipticity} establishes generic hypoellipticity as in (2). We conclude in Section \ref{subsec:genericQuadPassthrough} with the proof of dynamical condition (3). 

\begin{remark} 
	In this section, we will use the word \emph{generic} to mean both ``open and dense'' \emph{and} ``full Lebesgue measure''. In practice however the conditions (1) -- (3) are easily seen to be open conditions in the coefficients $b \in \Bc$, and so for us it will be enough to check (1) -- (3) hold on a full Lebesgue-measure set. 
\end{remark}



\subsection{The constraint class $\Bc$}\label{subsec:constraintClass3}

\begin{proposition}\label{prop:constraintClass2}
	It holds that $b \in \Bc$ if and only if 
	\begin{align}
		b^i_{jk} = b^i_{kj} \quad & \text{for all} \quad 1 \leq i,j,k \leq d \, , \\
		\label{eq:noDoubleEntries2}
		b^{i}_{ii} = b^i_{ij} = b^i_{jj} = 0 \quad &\text{for all} \quad 1 \leq i, j \leq d \, , \quad \text{ and } \\ \label{eq:jacobiIdentity}
		b^i_{jk} + b^j_{ik} + b^k_{ij} = 0 \quad &\text{for all} \quad 1\leq i,j,k \leq d \,. 
	\end{align}
\end{proposition}

	\begin{remark} \label{eq:specifyFromJacobiIdentity}
		In practice, an element $b \in \Bc$ is specified by the coordinates $b^i_{jk}$ and $b^{j}_{ik}$ for all triples $\{ i, j, k\} \subset \{ 1, \dots, d\}$ of distinct coordinates, since $b^k_{ij}$ is specified in terms of the other two. This idea will be used several times in the coming proofs. 
	\end{remark}

\begin{proof}[Proof of Proposition \ref{prop:constraintClass2}]
	Below, we focus on checking that $b \in \Bc$ implies \eqref{eq:noDoubleEntries2} and \eqref{eq:jacobiIdentity}; the converse implication is straightforward to check and amounts to reading the following proof in reverse.  

	Equation \eqref{eq:divFree} gives
	\[\nabla \cdot B_b(x,x) = \sum_{i,j,k}\frac{\partial}{\partial x^i}b^i_{jk}x^jx^k = 2b^i_{ik}x^k=0.\] Matching $x^k$ coefficients then implies \[\sum_{i}b^i_{ik}=0 \text{ for each }1\leq k \leq d.\]
	From \eqref{eq:energyPres}, 
	\[x\cdot B_b(x,x) = \sum_{i,j,k} b^i_{jk}x^ix^jx^k=0.\] Matching coefficients by permuting all combinations then gives \eqref{eq:jacobiIdentity} for all $i,j,k$, noting that one or more of $i,j,k$ may coincide. 
	Setting $i=j=k$ in \eqref{eq:jacobiIdentity} gives \[b^i_{ii}=0 \, , \] while $j=k$ gives \begin{align}\label{eq:partialSumFormula2}b^i_{jj}+2b^j_{ij}=0 \,. \end{align}
	Finally, condition \eqref{eq:vanishCoordAxes} imposes that
	\begin{align}
		\label{eq:noDoubleTerms}
		b_{jj}^i = 0 
	\end{align}
	for all $1 \leq i, j \leq d$, which in conjunction with 
	\eqref{eq:partialSumFormula2} implies \eqref{eq:noDoubleEntries2}. 
\end{proof}

\subsection{Generic hyperbolicity along $\Ec$} \label{subsec:genercHyperbolicity3}


We now turn to the first of our generic dynamical conditions, hyperbolicity for the equilibria of $B$ along the coordinate axes $\Ec = \cup_{i = 1}^d \Span\set{e_i}$. 

\begin{proposition}\label{prop:genericHyperbolicity3}
	Assume $d \geq 4$. 
	For a full Lebesgue-measure set of $b \in \Bc$, it holds that $(DB_b)_{x}$ admits a single eigenvalue along the imaginary axis for all $x \in \Ec \setminus \{ 0 \}$. 
\end{proposition}
To prove Proposition \ref{prop:genericHyperbolicity3} it  suffices to check $(D_xB_b)_{e_i}$ for each $i \in \{ 1, \dots, d\}$. 
We directly compute
\[(DB_b)_{e_i}(e_j) = 2 \begin{pmatrix}
	b^1_{ij} \\ b^2_{ij} \\ \vdots \\ b^d_{ij}
\end{pmatrix} \, , \]
which for $i = 1$ (the other cases are argued similarly) gives 
\[
(DB_b)_{e_1} = 2 \begin{pmatrix}
	0 & 0 & 0 & 0 & \cdots  & 0 & 0 \\ 
	0 & 0 & b^2_{13} & b^2_{14} & \cdots & b^2_{1,d-1} & b^2_{1d} \\ 
	0 & b^3_{12} & 0 & b^3_{14} & \cdots & b^3_{1,d-1} & b^3_{1d} \\ 
	0 & b^4_{12} & b^4_{13} & 0 &  & b^4_{1,d-1} & b^4_{1d}  \\ 
	\vdots & \vdots & \vdots &  & \ddots & & \vdots  \\ 
	0 & b^{d-1}_{12} & b^{d-1}_{13} & b^{d-1}_{14} & \cdots & 0 & b^{d-1}_{1d}\\ 
	0 & b^d_{12} & b^d_{13} & b^d_{14} & \cdots & b^d_{1,d-1} & 0 
\end{pmatrix}
\]
There are exactly $(d-1)^2 - (d-1) = (d-1)(d-2)$ distinct entries of $(b^i_{jk})$ appearing above, with two representatives each of the triples $(1,j,k)$ for $j, k \in \{ 2, \dots, d\}$ (c.f. Remark \ref{eq:specifyFromJacobiIdentity}). So, as $b$ varies, the $(1,1)$-minor of $(DB_b)_{e_1}$ ranges over all possible \emph{hollow matrices} -- here, a matrix is called \emph{hollow} if its diagonal entries vanish. 

Thus, to prove Proposition \ref{prop:genericHyperbolicity3} it suffices to check the following. Below, for $m \geq 1$ we write $\Hc_m$ for the space of $m\times m$ hollow matrices. 
\begin{lemma}\label{lem:genericHollow}
	For any $m \geq 3$, there is a generic set of $A \in \Hc_m$ for which $\sigma(A) \cap i \R = \emptyset$. 
\end{lemma}

The proof of this lemma constitutes the remainder of Section \ref{subsec:genercHyperbolicity3}. 

\bigskip

First, we argue that a generic subset of $A \in \Hc_m$ have simple spectrum, i.e., all algebraic multiplicities are equal to one. This is a consequence of the following standard Lemma (see, e.g., \cite{gelfand1994hyperdeterminants}).

\begin{lemma}
	There exists a polynomial $\Disc : M_{m \times m}(\R) \to \R$ with integer coefficients (a.k.a. the \emph{discriminant} $\Disc(A)$ of a matrix $A$) with the property that $\Disc(A) \neq 0$ iff $A$ has simple spectrum. 
\end{lemma}
It is straightforward to check that $\{ \Disc(A) = 0\} \cap \Hc_m$ is a proper algebraic subvariety of $\Hc_m$, hence has Lebesgue measure zero and is closed with empty interior. 

    



Next, we recall that at matrices $A \in M_{m \times m}(\R)$ with simple spectrum, the eigenvalues $\lambda_i : M_{m \times m}(\R) \to \C, 1 \leq i \leq m$ vary in a real-analytic way (see, e.g., \cite[Chapter 2.1.1]{kato2013perturbation}). Form the function $\Phi : \mathcal{H}_m \to \R$ given by 
\[\Phi(A) = \prod_{1}^{m} \Re( \lambda_i(A)) \,, \]
noting that we seek to show $\{ \Phi \neq 0\}$ is a generic subset of $\mathcal{H}_m$. 
It is a standard fact in the structure theory of real-analytic varieties (c.f. \cite[Section 6.3]{krantz2002primer}) that the zero set of a non-constant real-analytic function on an open, connected domain is contained in a union of positive codimension real-analytic submanifolds of the domain. It follows that $\{ \Phi \neq 0\}$ is a generic subset of each open connected component $U$ of $\mathcal{H}_m \setminus \{ \Disc = 0\}$, so long as $\Phi|_U$ is non-constant (e.g., not identically equal to zero). 

We will prove something stronger, that $\{ \Phi \neq 0\}$ is dense in $\mathcal{H}_m \setminus \{ \Disc = 0\}$, from which the above argument will imply $\{ \Phi \neq 0\}$ is a generic subset of $\Hc_m$. Precisely: 

	\begin{lemma}\label{lem:pertStepHollow}
		Let $m \geq 3$. Let $A \in \Hc_m$ and assume  $A_\sym$ is invertible. Then, for any $\epsilon > 0$ there exists $\hat A \in \Hc_m$ with $\| A - \hat A\| < \epsilon$ such that $\sigma(\hat A) \cap i \R = \emptyset$. 
	\end{lemma}
	Here, $A_\sym := \frac12(A + A^\top)$ is the symmetrization of $A$. Invertibility of $A_\sym$ is an open, dense and full Lebesgue-measure condition on $\Hc_m$, and so this assumption does not pose a problem for the proof of Lemma \ref{lem:genericHollow}. 

The following is the main technical step in the proof of Lemma \ref{lem:pertStepHollow}. Its proof is deferred to the end of Section \ref{subsec:genercHyperbolicity3}. 

\begin{lemma}\label{lem:snapToHollow}
	For any $\eps, M > 0$ and $m \geq 2$ there exists $\delta > 0$ with the following properties. Let $A \in \mathfrak{sl}_m(\R)$ be such that 
	\begin{itemize}
		\item[(1)] $|v \cdot A v| < \delta$  for each  $i \in \set{1, \dots, m}$, and 
		\item[(2)] $\| A \|, \| A_\sym^{-1}\| \leq M$. 
	\end{itemize}
	Then, there exists an $m \times m$ orthogonal matrix $U$ such that $\hat A := U^\top A U \in \Hc_m$ and $\| \hat A - A\| < \epsilon$. 
\end{lemma}

\begin{proof}[Proof of Lemma \ref{lem:pertStepHollow} assuming Lemma \ref{lem:snapToHollow}]
	Let $A \in \Hc_m$ with $A_{\sym}$ invertible and let $\eps > 0$. If $\sigma(A) \cap i \R = \emptyset$ then there is nothing to do. Otherwise, fix $\delta > 0$ as in Lemma \ref{lem:snapToHollow} corresponding to $\epsilon / 2$ and $2M$ 
	where $M = \max\set{\| A\|, \| A_\sym^{-1}\| }$. Without loss, we may assume $\delta < \min\set{\eps / 2, M / 100}$. 
	
	Let now $A' \in \mathfrak{sl}_m(\R)$ be such that $\sigma(A') \cap i \R = \emptyset$ and $\| A - A'\| < \delta$. That such a perturbation $A'$ exists within $\mathfrak{sl}_m(\R)$ is straightforward\footnote{Note that here we use the fact that $m \geq 3$. If $m =2 $ the result is false, since the set of matrices in $\mathfrak{sl}_2(\R)$ with complex-conjugate eigenvalues is open.} to check from, e.g., the Jordan canonical form of $A$.  It is straightforward to check that $\| A'\| \leq M + \delta < 2 M$ and $m( A_\sym) \geq M^{-1} - \delta > (2 M)^{-1}$. Applying Lemma \ref{lem:snapToHollow}, there exists $\hat A \in \Hc_m$ orthogonally similar to $A'$ with $\| \hat A - A'\| < \epsilon / 2$, hence $\sigma(\hat A) \cap i \R = \emptyset$. Finally, we estimate 
	\[\| A - \hat A\| \leq \| A - A'\| + \| A' - \hat A\| \leq \delta + \frac{\eps}{2} < \eps \]
	as desired. 
\end{proof}

It remains to prove Lemma \ref{lem:snapToHollow}. The proof below is inspired by the methods of \cite{fillmore1969similarity}. We begin with the following Claim. 

\begin{claim}\label{cla:snapStep}
	For any $\eps, M > 0$ and $m \geq 1$ there exists $\delta > 0$ with the following property. Let $A \in \mathfrak{sl}_m(\R)$ be such that (1) $|v \cdot A v| < \delta$ and (2) $\| A \|, \| A_\sym^{-1}\| \leq M$. Then, there exists a vector $p \in \R^d$ with 
	\[\| p - v \| < \eps \qquad \text{ and } \qquad p \cdot A p = 0 \,. \]
\end{claim}

\begin{proof}[Proof of claim]
	Let us write $p = v + \eta w$ where $\eta \in \R$ and $w \in \R^m, |w| = 1$ are to be specified. The relation $p \cdot A p = 0$ simplifies to $p \cdot A_\sym p = 0$, which expands as 
	\begin{align} \label{eq:solveEtaQuadratic} v \cdot A_\sym v + 2 \eta ( v \cdot A_\sym w) + \eta^2 (w \cdot A_\sym w) = 0 \,; \end{align}
	of interest for us will be the solution
	\begin{align}\label{eq:defineEta}\eta & = \frac{-  (v \cdot A_\sym w) + \sqrt{(v \cdot A_\sym w)^2 - (w \cdot A_\sym w )(v \cdot A_\sym v)}}{ (w \cdot A_\sym w)} 
\end{align}
Choose $w$ to be an argmax of $w \mapsto |w \cdot A_\sym v|$ over the unit sphere $\set{\| w\| = 1}$ in $\R^m$, noting that by construction \[|v \cdot A_\sym w| = |w \cdot A_\sym v| = \| A_\sym v\| \geq \|A_\sym^{-1}\|^{-1} \geq M^{-1} \,.  \]  

Assume now that $|v \cdot A_\sym v| < \delta$, where $\delta$ will be taken smaller as we progress. For now, let us also assume that for this choice of $w$, one has $w \cdot A_\sym w \neq 0$, hence the RHS of \eqref{eq:defineEta} is defined. Then, 
\[\eta = \frac{v \cdot A_\sym w}{w \cdot A_\sym w} \left(-1 + \sqrt{1 - \frac{(v \cdot A_\sym v) (w \cdot A_\sym w)}{(v \cdot A_\sym w)^2}}\right) \, .
\]
We estimate 
\[\left|\frac{(v \cdot A_\sym v) (w \cdot A_\sym w)}{(v \cdot A_\sym w)^2}\right| \leq \frac{\delta M}{M^{-2}} = M^3 \delta \,. \] 
If $\delta \ll M^{-3}$, then the above LHS is $< 1$. Using that $\abs{-1 + \sqrt{1 - \alpha}} \leq |\alpha|$ for all $\alpha \in [-1,1]$, it follows that
\begin{align*} |\eta| & \leq \left| \frac{v \cdot A_\sym w}{w \cdot A_\sym w} \right| \cdot  \left| \frac{(v \cdot A_\sym v) (w \cdot A_\sym w)}{(v \cdot A_\sym w)^2} \right| \\ 
	& = \left| \frac{v \cdot A_\sym v}{v \cdot A_\sym w} \right| \leq M^{-1} \delta < \eps \end{align*}
where in the last step $\delta$ is taken yet smaller. 

In the case when $w \cdot A_\sym w = 0$, we have directly from \eqref{eq:solveEtaQuadratic} the simpler estimate
\[|\eta| = \frac12 \left|\frac{v \cdot A_\sym v}{v \cdot A_\sym w}\right| \leq \frac12 \delta M \, , \]
the RHS of which can be made $< \eps$ on taking $\delta < \eps M^{-1}$. 
\end{proof}
\begin{proof}[Proof of Lemma \ref{lem:snapToHollow}]
	The proof is induction in the dimension $m$. For the base case $m = 2$, let $A \in \mathfrak{sl}_2(\R)$, $M > 0$ and $\eps > 0$ be fixed as in the hypotheses of Lemma \ref{lem:snapToHollow}. Let $\delta > 0$ be as in Claim \ref{cla:snapStep} corresponding to this value of $M$ and to $\eps / 10$; from the assumption $|A e_1 \cdot e_1| < \delta$ let $p$ be such that $A p \cdot p = 0$ and $\| p - e_1 \| < \eps/10$ as in Claim \ref{cla:snapStep}. Fix an orthogonal matrix $U$ so that $U e_1$ is parallel to $p$ and define $\hat A = U^\top A U$, from which it follows that $e_1 \cdot \hat A e_1 = 0$. Note that since $\hat A \in \mathfrak{sl}_2(\R)$, $\hat A$ has trace zero and so it follows automatically that $e_2 \cdot\hat A e_2 = 0$, hence $\hat A \in \Hc_2$. The estimate $\| A - \hat A\| < \eps$ is now straightforward. 

	For the induction step, let $m > 2$ be fixed and assume the conclusions of Lemma \ref{lem:snapToHollow} hold for matrices in $\mathfrak{sl}_{m-1}(\R)$ with $\delta_{m-1} = \delta_{m-1}(\eps, M)$. 
	With $\eps, M > 0$ fixed, let $A \in \mathfrak{sl}_m(\R)$ satisfy assumptions (1) -- (3) as in the hypotheses of Lemma \ref{lem:snapToHollow} for some value of $\delta = \delta_m > 0$, to be taken smaller as we progress. 

	Taking $\delta_m$ small enough, depending only on $M$ and $\eps$, we can apply Claim \ref{cla:snapStep} to $v = e_1$ to yield $p \in \R^m$ for which $\| p - v \| < \eps / 20$ and $p \cdot A p = 0$. Fix an orthogonal matrix $U_1$ with $\| U_1 - \Id \| < \eps / 10$ such that $U_1 e_1$ is parallel to $p$ and form $\hat A' = U_1^\top A U_1$, noting that $\| \hat A' - A\| < \frac{\eps}{5}$ and that $\hat A'$ admits the block diagonal form 
	\[\hat A' = \begin{pmatrix}
		0 & \ast \\ \ast & A_{m-1}
	\end{pmatrix}\]
	where $A_{m-1} \in \mathfrak{sl}_{m-1}(\R)$. Here $\ast$ stands for $(m-1) \times 1$ and $1 \times (m-1)$ sub-blocks in $\hat A'$ which will not matter in the coming argument. Note that $A_{m-1}$ satisfies the estimates
	$\| A_{m-1} \| \leq \| \hat A'\| = \| A\| \leq M$ and\footnote{To see this, observe that \[\begin{pmatrix} 0 \\ v \end{pmatrix} \cdot \hat A' \begin{pmatrix} 0 \\ v \end{pmatrix} = v \cdot A_{m-1} v = v \cdot (A_{m-1})_\sym v  \] for all $v \in \R^{m-1}$.} $\|(A_{m-1})_\sym^{-1}\| \leq M$. Apply the induction hypothesis to obtain $\delta_{m-1} = \delta_{m-1}(\eps / 5, M)$, yielding $U_{m-1}$ orthogonal for which $\hat A_{m-1} = U_{m-1}^\top A_{m-1} U_{m-1}$ is hollow and $\| U_{m-1} - \Id\| < \eps / 5$. Finally, set 
	\[U = U_1 \begin{pmatrix}
		1 & 0 \\ 0 & U_{m-1} 
	\end{pmatrix}\] 
	so that 
	\[\hat A = U^\top A U = \begin{pmatrix}
		1 & 0 \\ 0 & U_{m-1}^\top  
	\end{pmatrix} \hat A' \begin{pmatrix}
		1 & 0 \\ 0 & U_{m-1} 
	\end{pmatrix} = \begin{pmatrix}
		0 & \ast \\ \ast & U_{m-1}^\top A_{m-1} U_{m-1}
	\end{pmatrix}\]
	is hollow and satisfies $\| \hat A - A \| < \eps$, as desired. 
\end{proof}

\subsection{Generic hypoellipticity}\label{subsec:genericHypoellipticity}

Next, we turn to the second generic dynamical condition: a generic parabolic H\"ormander condition. It suffices to prove the following. 

\begin{proposition}\label{prop:genHypoelliptic2}
    For a generic set of $b \in \Bc$, it holds that for any distinct $i,j \in \{ 1, \dots, d\}$, for any $d \times d$ matrix $M$, and for any $\eps \in \R$, one has 
    \begin{align} \label{eq:LieAlgsuff3}\operatorname{Lie} \{ e_i, e_j, [B_b + \eps Mx, e_i], [B_b + \eps Mx,e_j] \} = \R^d\end{align}
    at all $x \in \R^d$.
\end{proposition}


\begin{proof}[Proof of Proposition \ref{prop:genHypoelliptic2}]
 In what follows, we obtain a generic set of $b \in \Bc$ for which \eqref{eq:LieAlgsuff3} holds for $i = 1, j = 2$ and for all $\eps$ and $M$. The proof for general $i,j$ is identical up to relabeling of coordinates. Since there are only finitely many possible combinations of $i, j$, it follows on taking a finite intersection that for generic $b \in \Bc$, all pairs of $i,j$ satisfy \eqref{eq:LieAlgsuff3} simultaneously. 

Below we will define inductively a set of vectors $v_1, v_2, \dots, v_d, v_i = v_i(b)$, with entries given by polynomials in the coefficients of $b \in \Bc$, for which 
\begin{align}\label{eq:LieAlgMembership}\{ v_1, \dots, v_d\} \subset \operatorname{Lie}\{ e_i, e_j, [B_b + \eps Mx, e_i], [B_b + \eps Mx,e_j] \}\end{align}
holds at all $x \in \R^d$, independently of $\eps \in \R$ or the matrix $M$. 
These vectors will be constructed in such a way so as to guarantee the existence of some $b_* \in \Bc$ for which 
\[v_i(b_*) = e_i \qquad \text{ for all } \quad 1 \leq i \leq d \, ,\]
from which it will follow that the polynomial 
\[G(b) = \det \begin{pmatrix}
    | &  & | \\ 
    v_1(b) & \cdots & v_d(b) \\ 
    | &  & | 
\end{pmatrix}\]
satisfies $G(b_*) = 1$. It follows that $\{ G = 0\} \subset \Bc$ is a proper algebraic subvariety, hence its complement has full Lebesgue measure, and Proposition \ref{prop:genHypoelliptic2} follows. 

The $v_i$ are defined inductively as follows: 

\begin{align*}
    v_1 & = e_1 \\ 
    v_2 & = e_2 \\ 
    v_3 & = [v_2, [v_1, B_b  + \eps Mx]] \\ 
    v_4 & = [v_3, [v_2, B_b  + \eps Mx]] - b^1_{23} v_1 \\ 
    v_5 & = [v_4, [v_3, B_b  + \eps Mx]] - b^1_{34} v_1 - b^2_{34} v_2 \\ 
    \vdots & \\ 
    v_{i + 2} & = [v_{i+1}, [v_i, B_b  + \eps Mx]] - \sum_{j = 1}^{i-1} b^j_{i, i+1} v_j \, , \qquad i \leq d - 2 \,. 
\end{align*}
Observe that at each stage, it holds that $v_i$ depends only on $b$ and not on $\eps, M$ or $x$. This follows by induction and that if $v, w$ are any two constant vector fields in $\R^d$, then 
	\[[v,[w,B_b + \eps Mx]](x) = \begin{pmatrix}
		v^\top B_b^1 w \\ \vdots \\ v^\top B_b^d w
	\end{pmatrix} \,. \]
Since the $v_i(b)$'s are generated as linear combinations of $e_1, e_2$ and brackets thereof with $B$, the relation \eqref{eq:LieAlgMembership} is immediate. 

We now choose the coefficients $b^i_{jk}$ of $b_*$ as follows. To start we compute 
\begin{align*}
    v_3 = [v_2, [v_1, B_b]] = B_b(v_1, v_2) = \begin{pmatrix}
        0 \\ 
        0 \\ 
        b^3_{12} \\ 
        b^4_{12} \\ 
        \vdots \\ 
        b^d_{12} 
    \end{pmatrix} \,. 
\end{align*}
We will set $b^3_{12} = 1$ and $b^4_{12} = \dots = b^d_{12} = 0$, noting that so far we have specified only those coefficients $b^i_{jk}$ with indices drawn from \[S_3 = \{ \{1, 2, i\} : i \geq 3\} \,, \] 
c.f. Remark \ref{eq:specifyFromJacobiIdentity}. 
With this choice, we have $v_3(b_*) = e_3$ as desired. 

Inductively, assume $v_j(b_*) = e_j$ holds for all $j \leq i + 1$, where $i \leq d -2$, and that along the way we have specified coefficients $b^i_{jk}$ of $b_*$ belonging to $S_3 \cup \cdots \cup S_{i+1}$ where 
\[S_j = \{ (j-2, j-1, k) : j-1 < k \leq d \} \,. \]
We seek to show that under these conditions, one can specify coefficients with indices in $S_{i + 2}$, a set disjoint from $S_3 \cup \dots \cup S_{i+1}$, for which $v_{i + 2}(b_*) = e_{i+2}$. For this we compute 
\[B_b(e_i, e_{i + 1}) = \begin{pmatrix}
    b^1_{i, i+1} \\ \vdots \\ b^{i-1}_{i, i+1} \\ 0 \\ 0 \\ b^{i+2}_{i, i+1} \\ b^{i+3}_{i, i+1} \\ \vdots \\ b^d_{i, i+1} 
\end{pmatrix} \, , \quad \text{ hence } \quad v_{i+2} = \begin{pmatrix}
    0 \\ \vdots \\ 0 \\ 0 \\ 0 \\ b^{i+2}_{i, i+1} \\ b^{i+3}_{i, i+1} \\ \vdots \\ b^d_{i, i+1} 
\end{pmatrix} \,. \]
The coefficients appearing in $v_{i+2}$ belong to $S_{i+2}$, as promised, and shall be set so that $b^{i+2}_{i, i+1} = 1$ and $b^j_{i, i+1} = 0$ for all $i+2 < j \leq d$. With this assignment it holds that $v_{i+2}(b_*) = e_{i+2}$, completing the induction step. 
\end{proof}

\subsection{Passing through $\mathcal{K}$}\label{subsec:genericQuadPassthrough}

The last step is to check generic nonexistence of invariant sets in $\Kc$ for the deterministic flow $\dot x = B_b(x,x)$. 
\begin{proposition}\label{prop:genericPassthrough}
	For a generic subset of $b \in \Bc$, there exists $K \geq 1$ such that trajectories $x: \R \to \R^d$ of the deterministic flow $\dot x = B_b(x,x)$ have the property that 
	\begin{align}\label{eq:passthroughDerivFormulation}
		x(t) \in \Kc \setminus \Ec \quad \Rightarrow \quad \exists i \in \set{1, \dots, K} \text{ such that } \frac{d^i}{dt^i} x(t) \notin \Kc \,.
	\end{align}
	In particular, there are no invariant subsets of $\Kc \setminus \Ec$ for the deterministic flow $\dot x = B(x,x)$.  
\end{proposition}

	

The argument we present is based on the following version of the Transversality Theorem (see, e.g., \cite[Section 2.3]{guillemin2010differential}). 

\begin{definition}
	Let $X, Y$ be smooth manifolds 
	and let $Z \subset Y$ be a smooth submanifold, all of which are assumed boundaryless. We say that a mapping $F : X \to Y$ is transversal to $Z$, written $F \pitchfork Z$, if for all  we have that 
	\begin{align}\label{eq:transDefn3}
		\text{for all } x \in F^{-1}(Z) \, , \quad  {\rm Image}(DF_x) + T_{F(x)} Z = T_{F(x)} Y \,. 
	\end{align}
	Here, we use the convention that $F \pitchfork Z$ if $F^{-1}(Z) = \emptyset$, in which case the relation \eqref{eq:transDefn3} is vacuously true. 
\end{definition}

\begin{theorem} \label{thm:transversality}
	Let $F : S \times X \to Y$ be smooth and let $Z \subset Y$ be a submanifold (all boundaryless). Suppose that $F \pitchfork Z$. Then for a.e. $s \in S$, the mapping $f_s := F(s, \cdot), f_s : X \to Y$ is transversal to $Z$. 
\end{theorem}
We will apply Theorem \ref{thm:transversality} to the mapping $(b, x) \mapsto \B^{(2)}_b(x)$ defined by 
	\[\B^{(2)}(b,x) = (x, B_b(x,x), 2 B_b(B_b(x,x), x)) \,, \]
	viewed as a smooth mapping $\Bc \times \R^d \to (\R^d)^3$. 
The mapping $\B^{(2)}$ has the property that given $b \in \Bc$ and a trajectory $x: \R \to \R^d$ of the ODE $\dot x = B_b(x,x)$, it holds that
	\[\B^{(2)}_b(x(t)) = (x(t), \dot x(t), \ddot x(t)) \,.\]

As we will show below, the following is sufficient to prove Proposition \ref{prop:genericPassthrough}. 
	\begin{proposition}\label{prop:quadTransversality2}
		Let $\mathcal{S} \subset \set{1, \dots, J}$ have cardinality $q \geq 3$. Then, 
		the mapping 
		\[\B^{(2)} : \Bc \times(\R^{d} \setminus \Ec_{q-1} ) \to \R^{3d}\] is transversal to $(\R^{\mathcal{S}})^3$. 
	\end{proposition}
Above and in what follows we have and shall continue to use the following notation:  
\begin{itemize}
	\item[(1)] Given a subset $\mathcal{S} \subset \{ 1, \dots, d\}$ we shall write 
	\[\R^{\mathcal{S}} = \Span\set{e_i : i \in \mathcal{S}} \, .\]
	\item[(2)] For $j \leq d$ we shall write 
	\[\Ec_j = \bigcup_{\substack{\mathcal{S} \subset \set{1, \dots, d} \\ |\mathcal{S}| = j}} \R^{\mathcal{S}}\]
	for the union over all coordinate hyperplanes of dimension $j$. 
	Observe that when $|\mathcal{S}| = q \geq 2$, it holds that $\R^{\mathcal{S}} \setminus \Ec_{q -1}$ is the set of $x \in \R^{\mathcal{S}}$ for which $x^i \neq 0$ for all $i \in \mathcal{S}$. 
	\item[(3)] When $b \in \Bc$ is fixed we will abuse notation somewhat and write \begin{gather*}
		\dot x = \dot x(b, x) = B_b(x,x) \, , \\ 
		\ddot x = \ddot x(b, x) = 2 B_b(B_b(x,x),x) \, , 
	\end{gather*} 
	so that $\B^{(2)}(b, x) = (x, \dot x, \ddot x)$.
\end{itemize}
 
\subsection*{Proof of Proposition \ref{prop:genericPassthrough} assuming Proposition \ref{prop:quadTransversality2}}

The proof of Proposition \ref{prop:quadTransversality2} is deferred till the end of Section \ref{subsec:genericQuadPassthrough}, and for now we will turn to how the proof of Proposition \ref{prop:genericPassthrough} is to be completed. 

We begin with the following claim. 

\begin{claim}\label{cla:passThroughStep}
	Let $\mathcal{S} \subset \set{1, \dots, J}$ have cardinality $q \geq 2$. 
	There is a generic subset of $b \in \Bc$ with the property that 
	\begin{align} \label{eq:stickOut3} x \in \R^{\mathcal{S}} \setminus \Ec_{q-1} \quad \Rightarrow \quad (\dot x, \ddot x) \notin \R^{\mathcal{S}} \,. \end{align}
\end{claim}
\begin{proof}
	If $|\mathcal{S}| \geq 2$ it is straightforward to check that $\dot x \notin \R^{\mathcal{S}}$ for $x \in \R^{\mathcal{S}} \setminus \Ec_1 = \R^{\mathcal{S}} \setminus \Ec$ if $b^i_{jk} \neq 0$ for $j, k$ the two distinct elements of $\mathcal{S}$ and for some $i \notin \mathcal{S}$. This is a generic condition in $b$ and \eqref{eq:stickOut3} follows. 

	If $|\mathcal{S}| \geq 3$, Proposition \ref{prop:quadTransversality2} and the Transversality Theorem imply that for a full Lebesgue-measure subset of $b \in \Bc$ it holds that $B^{(2)}_b(\cdot) = \B^{(2)}(b, \cdot) : (X^{\mathcal{S}} \setminus \Ec_{q -1}) \to \R^{3d}$ is transversal to $(\R^{\mathcal{S}})^3$. Assuming transversality as above and when $J < \frac{2d}{3}$, we will now check that 
\[x \in \R^{\mathcal{S}} \quad \Rightarrow \quad (\dot x, \ddot x) \notin (\R^{\mathcal{S}})^2 \, .\] Otherwise, transversality and existence of $x \in (B^{(2)}_b)^{-1} (\R^{\mathcal{S}})^3$ would imply
\[\operatorname{Image}((D B^{(2)}_b)_x) + T_{B^{(2)}_b(x)} (\R^{\mathcal{S}})^3 = \R^{3d} \, , \]
from which $3d \leq 3 q + d \leq 3 J + d$ and $\frac{2d}{3} \leq J$, a contradiction. 
\end{proof}

\begin{proof}[Proof of Proposition \ref{prop:genericPassthrough} assuming Proposition \ref{prop:quadTransversality2}]

We will prove \eqref{eq:passthroughDerivFormulation} in the following equivalent form: 
\begin{gather} \label{eq:trajectoryFormulation3} x(t_0) \in \Kc \setminus \Ec \quad \text{ for some } t_0 \in \R \quad \Rightarrow \\ \forall \eps > 0 \text{ there exists } t \in (t_0, t_0 + \eps) \text{ such that } x(t) \notin \Kc \,. 
\end{gather} 
That \eqref{eq:trajectoryFormulation3} and \eqref{eq:passthroughDerivFormulation} are equivalent is a consequence of analyticity of $t \mapsto x(t)$ and a compactness argument.

Let $b \in \Bc$ be generic in the sense of Claim \ref{cla:passThroughStep} all $\mathcal{S} \subset \set{1, \dots, J}$ of cardinality $\geq 2$. Let $x : \R \to \R^d$ be a trajectory of $\dot x = B_b(x,x)$ and assume $x(t_0) \in \Kc \setminus \Ec$ for some $t_0 \in \R$. 

Let $\mathcal{S}_0$ be the set of nonzero coordinates of $x(t_0)$, noting $|\mathcal{S}_0| \geq 2$. By Claim \ref{cla:passThroughStep}, $(\dot x, \ddot x)(t_0) \notin \R^{\mathcal{S}_0}$. In particular, $\exists t_1 \in (t_0, t_0 + \frac{\eps}{J})$ such that the set $\mathcal{S}_1$ of nonvanishing coordinates of $x(t_1)$ is strictly larger, i.e., 
\[\mathcal{S}_1 \supsetneq \mathcal{S}_0 \, .\]

If $x(t_1) \notin \Kc$ then we are done. Otherwise, $\R^{\mathcal{S}_1} \subset \mathcal{K}$, and repeating the argument of the previous paragraph it holds that $\exists t_2 \in (t_1, t_1 + \frac{\eps}{J})$ such that the set of nonvanishing coordinates $\mathcal{S}_2$ of $x(t_2)$ strictly contains $\mathcal{S}_1$. Repeating this argument inductively, we obtain a sequence of times $t_1 < t_2 < \dots < t_r$ with $t_j - t_{j-1} < \frac{\eps}{J}$ for which the sets of nonvanishing coordinates $\mathcal{S}_j$ of $x(t_j)$ satisfy
\[\mathcal{S}_0 \subsetneq \mathcal{S}_1 \subsetneq \cdots \subsetneq \mathcal{S}_{j-1} \subsetneq \mathcal{S}_j \subsetneq \cdots  \, . \]
Since $|\mathcal{S}_{j}| \geq |\mathcal{S}_{j-1}| + 1 \geq \dots \geq |\mathcal{S}_0| + j \geq j + 2$, this procedure must terminate at some finite stage $r \leq d - 2$, resulting in $t_r > 0$ and $x(t_r) \notin \Kc$. This completes the proof. 
\end{proof}

Before proceeding, we comment on some aspects of the above proofs. 

\begin{remark} \label{rem:d/3} \ 
	\begin{itemize}
		\item[(1)] The constraint $J < \frac{2d}{3}$ comes up only in the proof above of Claim \ref{cla:passThroughStep}. Indeed, if instead one were to work with 
		\[\B^{(k)}(b, x) = (x, \dot x, \ddot x, x^{(3)}, \dots, x^{(k)})\]
		taking values in the $k$-th iterated tangent bundle $T^{(k)} \R^d \cong \R^{(k + 1) d}$, then transversality of $\B^{(k)}$ and $(\R^{\mathcal{S}})^{k + 1}$ would imply $(B^{(k)}_b)^{-1} (\R^{\mathcal{S}})^{k+1} = \emptyset$ when $J < \frac{k d}{k + 1}$ by the same dimension-counting argument as before. 
		
	 On the one hand, there are $2{d \choose 3}$ independent degrees of freedom in $\Bc$, which leaves ample-enough degrees of freedom to prove transversality when $2 {d \choose 3} > (k + 1) d$, i.e., for $k \approx d^2 / 3$ when $d$ is large. On the other hand, the number of terms in $D_b \B^{(k)}$ grows rapidly as $k$ increases, and the authors are unaware of how to proceed even for $k = 3$. 
		

		\item[(2)] There is a small subtlety here in the use of the second-derivative section $\B^{(2)}$: we do not, in fact, establish any information on $\dot x(t), \ddot x(t)$ when $x(t) \in \Kc$, but rather, that some higher derivative $\frac{d^j}{dt^j} x(t)$ is not parallel to $\Kc$. This is due to the fact that derivatives are `expended' in passing out of the nested family of degenerate subsets $\Ec_2, \Ec_3, \dots, \Ec_{J-1}$ as indicated at the end of the proof of Proposition \ref{prop:genericPassthrough}. 
		
		The following synthetic model illustrates this point: consider the vector field 
		\[V(z^1, \dots, z^d) = (1, z^1, z^2, \dots, z^{d-1})  \]
		for which the initial condition $z(0) = (0,\dots, 0)$ has the trajectory $z(t) = (t, t^2 / 2!, t^3 / 3!, \dots, t^d / d!)$. Even though $V(z) \notin \mathfrak{S}_\ell$ for all $z \in \mathfrak{S}_\ell \setminus \mathfrak{S}_{\ell - 1}$, the trajectory $z(t)$ takes on higher powers of $t$ to `climb' up the chain of degenerate sets $\mathfrak{S}_1 \subset \mathfrak{S}_2 \subset \dots \subset \mathfrak{G}_{d-1}$.

	\end{itemize}
\end{remark}

\subsubsection*{Proof of Proposition \ref{prop:quadTransversality2}}

Turning to the proof of Proposition \ref{prop:quadTransversality2}, let $3 \leq q \leq J$; without loss, let us take $\mathcal{S} = \set{1, \dots, q}$. We compute 
\begin{align*}
	D_{(b, x)}\mathbb{B}_b^{(2)}(c,y) 
	=   
	\bigg( & y,B_c(x,x)+2B_b(x,y),
	\\
	& 2B_c(B_b(x,x),x)+2B_b(B_c(x,x),x) \\ 
	& + 4 B_b (B_b(x,y), x) + 4 B_b(B_b(x,x), y) \bigg) \, \\ 
	& = ( y, D_b \dot x(c) + D_x \dot x(y) , D_b \ddot x(c) + D_x \ddot x(y)) \,. 
\end{align*}
We seek to show that for $b \in \Bc$ (perhaps subject to a generic condition, to be imposed later on) and for any $x \in \R^{\mathcal{S}} \setminus \Ec_{q-1}$ we have that for all $(s,t,u)\in \R^{3d}$  there exist $(c,y) \in \Bc \times \R^d$ and $v_1,v_2,v_3 \in \R^{\mathcal{S}}$ such that
\begin{align}\label{eq:objectiveB3}
	D_{(b, x)}\mathbb{B}^{(2)}(c,y) + (v_1,v_2,v_3) = (s,t,u).
\end{align}
Throughout we set $y=s$, $v_1 = 0$, which takes care of the first entry in \eqref{eq:objectiveB3}. 
It remains to treat the second and third entries.


Let $\Pi_{\leq q}$ denote orthogonal projection in $\R^d$ to $\Span\set{e_1, \dots, e_q}$ and $\Pi_{> q} = \Id - \Pi_{\leq q}$. Let $\bar t = t - D_x \dot x(y), \bar u = u - D_x \ddot x(y)$. Below, we will specify $c \in \Bc$ so that 
\[\Pi_{> q} D_b \dot x(c) = \Pi_{> q} \bar t \,, \qquad \Pi_{> q} D_b \ddot x(c) = \Pi_{> q} \bar u \, , \]
whereupon we will set 
\[v_2 = \Pi_{\leq q} \left(\bar t - D_b \dot x(c)\right) \, , \qquad v_3 = \Pi_{\leq q} \left(\bar u - D_b \ddot x(c)\right) \, \]
and \eqref{eq:objectiveB3} will follow. 

For $\set{k, m, p} \subset \{ 1, \dots, q\}$ distinct, let\footnote{Above, we abuse notation and treat $\Bc^q_{(k; m, p)}$ as a subspace of the tangent fiber $T_b \Bc \cong \Bc$ at some $b \in \Bc$, so that the derivation $\partial / \partial b^i_{k,m}$ is identified with $c = (c^j_{lm}) \in T_b \Bc$ with $c^j_{ln} = 1$ if $(j, l,n) = (i,k,m)$ and zero otherwise. } 
\[\Bc^q_{(k; m, p)} = \Span\set{\frac{\partial}{\partial b^i_{k,m}}, \frac{\partial}{\partial b^i_{k,p}} }_{i \in \set{q + 1, \dots, d}} \subset \Bc \,.  \] 
Consider the linear operator 
\[M : \Bc^q_{(k; m, p)} \to \Span\set{e_{q + 1}, \dots, e_d}^2\]
given for $c \in \Bc^q_{(k; m, p)}$ by 
\[M(c) = (\Pi_{> q} D_b \dot x(c), \Pi_{> q} D_b \ddot x(c)) \, . \]

To complete the proof of Proposition \ref{prop:quadTransversality2} it suffices to check that $M$ as above is invertible for \emph{some} choice of $\set{k, p, m} \subset \{ 1, \dots, q\}$. To see how this is done, we compute 
\begin{align*}
	D_b \dot x^i\left( \frac{\partial}{\partial b^j_{km}}\right) &= 2\delta^i_jx^kx^m,
	\\
	D_b \ddot x^i\left( \frac{\partial}{\partial b^j_{km}} \right)  &= 2\delta^i_j\sum_{k_2,k_3} (b^k_{k_2,k_3}x^m + b^m_{k_2,k_3}x^k)x^{k_2}x^{k_3}
	+ \sum_{k_1}4b^i_{j,k_1}x^{k_1}x^kx^m
	\\
	&=2\delta^i_j(\dot x^kx^m+\dot x^mx^k) + (\partial_{x^j}\dot x^i) (x^kx^m) \, , 
\end{align*}
where $\delta^i_j$ denotes the Kronecker delta. In matrix form, $M$ square of dimensions $2 (d - q) \times 2(d-q)$, expressed as
\begin{align*}
	\left(
	\begin{matrix}
		2x^kx^m & 2x^kx^p & 0 & 0 & \hdots
		\\
		2\dot x^kx^m+2\dot x^mx^k & 2\dot x^kx^p + 2\dot x^px^k & \partial_{x^{q+2}}\dot x^{q+1}(2x^kx^m) & \partial_{x^{q+2}}\dot x^{q+1}(2x^kx^p) & \hdots
		\\
		0 & 0 & 2x^kx^m & 2x^kx^p & \hdots
		\\
		\partial_{x^{q+1}}\dot x^{q+2}(2x^kx^m) & \partial_{x^{q+1}}\dot x^{q+2}(2x^kx^p) & 2\dot x^kx^m+2\dot x^mx^k & 2\dot x^kx^p+2\dot x^px^k & \hdots
		\\
		\vdots & \vdots & \vdots & \vdots & \hdots
		\\
		0 & 0 & 0 & 0 & \hdots
		\\
		\partial_{x^{q+1}}\dot x^d(2x^kx^m) & \partial_{x^{q+1}}\dot x^d(2x^kx^p) & \partial_{x^{q+2}}\dot x^d(2x^kx^m) & \partial_{x^{q+2}}\dot x^d(2x^kx^p) & \hdots
	\end{matrix}\right.                
	\\
	\left.
	\begin{matrix}
		\hdots & 0 & 0
		\\
		\hdots & \partial_{x^{d}}\dot x^{q+1}(2x^kx^m) & \partial_{x^{d}}\dot x^{q+1}(2x^kx^p)
		\\
		\hdots & 0 & 0
		\\
		\hdots & \partial_{x^d}\dot x^{q+2}(2x^kx^m) & \partial_{x^d}\dot x^{q+2}(2x^kx^p)
		\\
		\hdots & \vdots & \vdots
		\\
		\hdots & 2x^kx^m & 2x^kx^p
		\\
		\hdots &2\dot x^kx^m+2\dot x^mx^k & 2\dot x^kx^p+2\dot x^px^k
	\end{matrix}\right).
\end{align*}
where the basis for the domain (i.e., the columns of $M$) is enumerated as 
\[\set{\frac{\partial}{\partial b^1_{k,m}}, \frac{\partial}{\partial b^1_{k,p}}, \frac{\partial}{\partial b^2_{k,m}}, \frac{\partial}{\partial b^2_{k,p}}, \dots}\]
and the basis for the codomain (i.e., the rows of $M$) is enumerated as 
\[\set{(e_{q + 1}, 0), (0,e_{q + 1}), (e_{q + 2}, 0), (0,e_{q + 2}), \dots, (e_d, 0), (0,e_d)} \,. \]

Writing  $R_{11}, R_{12}, R_{21}, R_{22}, \dots, R_{(d - q),1}, R_{(d-q),2}$ for the $2(d - q)$ rows of the matrix $M$, let us perform the sequence of row replacements
\begin{align*} R_{22} & \mapsto R_{22} - (\partial_{x^{q+1}} \dot x^{q + 2}) R_{11} \\ 
R_{32} & \mapsto R_{32} - (\partial_{x^{q+1}} \dot x^{q + 3}) R_{11} \\ 
& \vdots \\
R_{(d-q), 2} & \mapsto R_{(d - q), 2} - (\partial_{x^{q+1}} \dot x^{d}) R_{11} \,. 
\end{align*}
By inspection, this series of row replacements eliminates all nonzero entries of $M$ in rows $R_{22}, R_{32}, \dots, R_{(d-q),2}$, and since $R_{11}$ has nonzero entries only in the first two columns, the remaining columns of $M$ are unaltered. Repeating this procedure, we successively eliminate all entries of $M$ off the main $2 \times 2$ block diagonal, hence $M$ row reduces to
\begin{align*}
	\begin{pmatrix}
		D & \hdots & 0
		\\
		\vdots & \ddots & \vdots
		\\
		0 & \hdots & D 
	\end{pmatrix}
\end{align*}
where $D$ is the $2 \times 2$ matrix given\footnote{Note that $D$ does not depend on $i$.} by
\begin{align*}
	D = \frac{\partial(\dot x^i,\ddot x^i)}{\partial (b^i_{k,m},b^i_{k,p})} 
	=
	\begin{bmatrix}
		2x^kx^m & 2x^kx^p
		\\
		2\dot x^kx^m + 2\dot x^mx^k & 2\dot x^kx^p+2\dot x^px^k
	\end{bmatrix} \,. 
\end{align*}
Thus, $M$ is invertible iff $D$ is invertible iff the expression 
\begin{align*}
	\det(D) = 4(x^k)^2(\dot x^px^m-\dot x^mx^p) 
\end{align*}
is nonvanishing. 
By this argument, we have reduced Proposition \ref{prop:quadTransversality2} to checking the following. 
\begin{claim}\label{cla:detDnonvanishing}
	There is an open, dense and full Lebesgue-measure set of $b \in \Bc$ with the following property: for any $x \in \R^{\mathcal{S}} \setminus \Ec_{q-1}$ there exists $\{ k, m, p\} \subset \{ 1, \dots, q\}$ such that $\det(D) \neq 0$. 
\end{claim}

The following lemma will be used in the proof of Claim \ref{cla:detDnonvanishing}, the proof of which is deferred for now. 

\begin{lemma}\label{lem:equilibria}
	For generic $b \in \Bc$, it holds that 
	\[\{ B_b(x,x) = 0\} \cap \R^{\mathcal{S}} \subset \Ec\]
	for all $\mathcal{S} \subsetneq \set{1,\dots, d}$. 
\end{lemma}

\begin{proof}[Proof of Claim \ref{cla:detDnonvanishing}]
	{\color{red}   }
	
	Let $b \in \Bc$ be a member of the generic set obtained in Lemma \ref{lem:equilibria}. 
In pursuit of a contradiction, the assumption that $\det(D) = 0$ for all $k,m,p \in \set{1, \dots, q}$  implies 
\[\dot x^{k_1} x^{k_2} = x^{k_1} \dot x^{k_2} \]
for all $1 \leq k_1, k_2 \leq q$, hence  
\begin{align}\label{eq:logDersAllEqual}
	\frac{\dot x^1}{x^1} = \cdots = \frac{\dot x^q}{x^q} \, . 
\end{align}
Let 
\[	X = x^1\cdots x^{j-1}\dot x^jx^{j+1}\cdots x^q \, , 
\]
noting that the right-hand side is independent of $j$ by \eqref{eq:logDersAllEqual}. Plugging in the definition of $\dot x^j$, it follows that for any $j$ one has
\begin{align*}
	X &=\sum_{k_1,k_2=1}^qb^j_{k_1,k_2}x^1\cdots x^{j-1}x^{j+1}\cdots x^q x^{k_1}x^{k_2} \, , 
\end{align*}
hence 
\begin{align*}
	\frac{\| x \|^2 X}{x^1 \dots x^q} =  \sum_{j, k_1, k_2} b^j_{k_1, k_2} x^j x^{k_1} x^{k_2} = x \cdot B_b(x,x) = 0 
\end{align*}
where in the last line we invoked the energy-preservation condition of our constraint class (c.f. Definition \ref{def:CCI}). Here, $\| x \|^2$ is the usual Euclidean norm. 

{In all, the contradiction hypothesis has implied  $X=0$. By assumption $x \notin \Ec_{q}$, hence $\dot x^j=0$ for each $1 \leq j \leq q$. Since $(x, \dot x, \ddot x) \in (\R^{\mathcal{S}})^3$ and $\mathcal{S} \subset \set{1,\dots, J}$, it now follows that $\dot x = 0$. Lemma \ref{lem:equilibria} implies $x \in \Ec$, hence the  contradiction. } 
\end{proof}








\begin{proof}[Proof of Lemma \ref{lem:equilibria}]
The plan is to show that for generic $b \in \Bc$ and for all $q = 2, 3, \dots, d-1$ that 
\begin{align}\label{eq:noBadEquilibria} B_b^{-1}(0) \cap (\R^{\mathcal{S}} \setminus \Ec_{q-1}) = \emptyset  \end{align}
for all $\mathcal{S} \subsetneq \set{1, \dots, d}$ of cardinality $q$. 

When $q = 2$ or $q = 3$ this can be checked by hand. Indeed, when $q = 2$, \eqref{eq:noBadEquilibria} follows for $\mathcal{S} = \set{1,2}$ under the generic condition $b^3_{12} \neq 0$. When $q = 3$ and $\mathcal{S} = \set{1,2,3}$, say, \eqref{eq:noBadEquilibria} follows when $b^1_{23} \neq 0$. 

It remains now to check \eqref{eq:noBadEquilibria} for $3 < q < d$. Fix such a $q$ and 
consider the mapping 
\[\B^{(1)} : \Bc \times  S^{d-1} \to T S^{d-1}\]
given by 
\[{\B}^{(1)} (b, x) = (x, B_b(x,x)) = (x, \dot x) \,. \]
Here, $S^{d-1}$ is the unit sphere in $\R^d$, and $T S^{d-1}$ is the tangent bundle to $S^{d-1}$ with fibers $T_x S^{d-1} = \Span{\set{x}}^{\perp}$.

Fix a set $\mathcal{S} \subset \set{1, \dots, d}$ with $|\mathcal{S}| = q$.  
 We will check that ${\B}^{(1)} \pitchfork \Sigma_{\mathcal{S}}$, where $\Sigma_{\mathcal{S}} \subset T S^{d-1}$ is the \emph{zero section} 
\[\Sigma_{\mathcal{S}} = \{ (x,0) : x \in (S^{d-1} \cap \R^{\mathcal{S}}) \setminus \Ec_{q-1}\} \]
of the submanifold $(S^{d-1} \cap \R^{\mathcal{S}}) \setminus \Ec_{q-1}$. From this it will follow from Theorem \ref{thm:transversality} that $B^{(1)}_b : S^{d-1} \to T S^{d-1}$ is transversal to $\Sigma_{\mathcal{S}}$ for generic $b \in \Bc$, where $B^{(1)}_b(x) := (x, B_b(x,x))$. Now, $B^{(1)}_b \pitchfork \Sigma_{\mathcal{S}}$ implies that the range of $B^{(1)}_b$ is disjoint from $\Sigma_{\mathcal{S}}$. Indeed, if $x \in (B^{(1)}_b)^{-1} (\Sigma_{\mathcal{S}})$, then 
\[D_x B^{(1)}_b(T_x S^{d-1} ) + T_{(x,0)} \Sigma_{\mathcal{S}} = T_{(x,0)} (T S^{d-1}) \, , \]
which is a contradiction since the RHS has dimension $2(d-1)$ while the LHS has dimension $\leq d-1 + q - 1 < 2 (d-1)$. The generic set of $\Bc$ in Lemma \ref{lem:equilibria} is now obtained by imposing each of the (finitely many) generic conditions corresponding to each $\mathcal{S}$. 

It remains to prove $\B^{(1)} \pitchfork \Sigma_{\mathcal{S}}$. For the sake of simplicity let us assume in the following transversality argument that $\mathcal{S} = \set{1,2, \dots, q}$ -- the argument in the general case is identical up to relabeling of indices.  Suppose $(b, x) \in (\B^{(1)})^{-1} (\Sigma_{\mathcal{S}})$. 
We seek to show that for all $(s, t) \in T_{(x, 0)}(T S^{d-1})$ there exist $(c, y) \in \Bc \times T_x S^{d-1}$ and $(v, 0) \in T_{(x, 0)} \Sigma_{\mathcal{S}}$ such that 
\[D_{(x, b)} \hat \B^{(1)}(c, y) + (v, 0) = (y, B_c(x,x) + 2 B_b(x, y)) + (v_1, 0) =  (s, t) \, . \]
To this end, set $v = 0, y = s$; it remains to specify $c \in \Bc$ so that 
\[B_c(x,x) = \bar t := t - 2 B_b (x, y) \,. \]
Writing $\bar t = (\bar t^i)$, we choose $c \in \Bc$ according to 
\[c^i_{12} = (x^1 x^2)^{-1} \bar t^i \quad \text{ for } i \geq 4\]
and 
\begin{align}
	c^{1}_{23} & = (x^2 x^3)^{-1} \bar t^1 \\ 
	c^2_{13} & = (x^1 x^3)^{-1} \bar t^2 \\ 
	c^3_{14} & = (x^1 x^4)^{-1} \bar t^3 \, .
\end{align}
Finally, all remaining coefficients $c^i_{jk}$ not specified above are set to 0. 
Note that division by $x^i x^j, i,j \in \mathcal{S} = \set{1, \dots, q}$ is defined, since $x \in \R^{\mathcal{S}} \setminus  \Ec_{q-1}$ implies $x^1, \dots, x^q \neq 0$. Also, note that the above coefficients only involve a single instance each of the triples $(1,2, i), i \geq 4$ and $(1,3,4)$, and exactly two instances of $(1,2,3)$. This ensures compatibility with the Jacobi relation \eqref{eq:jacobiIdentity}, c.f. Remark \ref{eq:specifyFromJacobiIdentity}. With these assignments, $B_c(x,x) = \bar t$, as desired, and the proof is complete.

\end{proof}

\section{Generic nonlinearities with time switching}
\label{sec:timedependent}
In this section we will prove Theorem~\ref{thm:mainswitched}. Throughout, we fix a coefficient $b_* \in \mathring{\Bc}$, an open ball $\mathcal{S} \subseteq \mathring{\Bc}$ containing $b_*$ that is compactly contained in $\mathring{\Bc}$, a damping matrix $A$ with $\mathrm{ker}A :=\Kc = \mathrm{span}\{e_1, \ldots, e_J\}$, and noise coefficients $\{\sigma_i\}_{i=1}^d$ satisfying $|\{i:\sigma_i \neq 0\}| \ge 2$. Then, $\{\Phi_n\}_{n\in \N}$ denotes the corresponding Markov chain defined in Section~\ref{sec:introswitching} with $I = [1/2,3/2]$, and we will write $\mathcal{P}$ for the associated Markov semigroup. While Theorem~\ref{thm:mainswitched} allows for any $J < d$, for concreteness we will assume $J = d-1$ and thus set $\Kc = \mathrm{span}\{e_1,\ldots, e_{d-1}\}$ for the entire section. For future use, we mention that the action of $\mathcal{P}$ on a measurable function $f:\R^d \to \R$ is given by 
\begin{equation}\label{eq:discretesemigroup}
	\mathcal{P}f = \int_{I} \int_{\mathcal{S}} P_t^b f \text{ }m_\mathcal{S}(\dee b) \dt,
\end{equation}
where $m_\mathcal{S}$ is normalized Lebesgue measure on $\mathcal{S}$ and $P_t^b$ is the Markov semigroup generated by \eqref{eq:SDE} with $A$ and $\{\sigma_i\}_{i=1}^d$ as above and $B=B_b$.

The sufficient condition that we will use to prove existence of a stationary measure for $\mathcal{P}$ is given by the lemma below and follows easily from a well-known existence criterion for discrete-time Markov chains. 

\begin{lemma} \label{lem:discreteexistence}
	Let $V(x) = 1+|x|^2$. Suppose that there exists $R > 1$ and $\alpha > 0$ such that for all $|x| \ge R$ there holds 
	\begin{equation} \label{eq:DiscreteKBgoal}
		\mathcal{P}^2 V(x) - V(x) \le -\alpha \frac{|x|}{\log(|x|)}.
	\end{equation} 
Then, $\mathcal{P}$ has at least one stationary measure.
\end{lemma} 

\begin{proof}
	By \eqref{eq:energyestimate}, there exists a constant $C>0$ such that for all $b \in \mathcal{S}$ and $t \in I$ we have 
	\begin{equation} \label{eq:Ptbenergy}
		P_t^b V(x) \le V(x) + C
	\end{equation}
	for every $x \in \R^d$.
	Therefore, from \eqref{eq:discretesemigroup} we have 
	\begin{equation}\label{eq:P2energy}
		\mathcal{P}^2 V(x) \le V(x) + 2C \quad \forall x\in\R^d.
	\end{equation}
	 Combining \eqref{eq:P2energy} and \eqref{eq:DiscreteKBgoal} it follows that 
	\begin{equation} \label{eq:P2general}
		\mathcal{P}^2 V(x) - V(x) \le -\alpha \frac{|x|}{1+\log(|x|+1)} + R + 2C \quad \forall x \in \R^d.
	\end{equation}
	Since the sub-level set
	$$
\left\{x \in \R^d: \frac{|x|}{1+\log(|x|+1)} \le K \right\}
	$$
	is compact for every $K > 0$, the bound \eqref{eq:P2general} implies that $\mathcal{P}^2$ has at least one stationary measure $\mu$. This follows from a straightforward generalization of the Krylov-Bogoliubov argument (see e.g. \cite[Lemma 2.7]{HerzogMatt24}). The probability measure 
	$$\nu = \frac{1}{2}\left(\mu + \mathcal{P}^* \mu\right)$$
	is then stationary for $\mathcal{P}$, where $\mathcal{P}^*$ denotes the dual action of $\mathcal{P}$ on measures.
\end{proof}

The plan for the remainder of this section is now to verify hypothesis \eqref{eq:DiscreteKBgoal} in Lemma~\ref{lem:discreteexistence}. To this end, we will require essentially two ingredients. First, we need a statement that plays the role of the dynamical assumption (3) used earlier to prove that trajectories starting away from $\Ec$ quickly experience damping, but which leverages the switched coefficients to be valid when just one mode is damped. This will be provided by Lemma~\ref{lem:SwitchingTransverse}. Second, in order to make use of Lemma~\ref{lem:SwitchingTransverse}, we need to show that the switches in the nonlinearity occur with sufficiently high probability when trajectories are not too close to $\Ec$. This will be the content of Lemma~\ref{lem:KdeltaTime} and follows from a suitable application of our earlier exit time estimates. 

We begin with our modified version of assumption (3), which will rely again on the Transversality Theorem. In what follows, we write $\varphi_b$ for flow map associated with the ODE $\dot{x} = B_b(x,x)$. Recall also the definition of the compact set $\Kd$ given at the beginning of Section~\ref{sec:exittimes}.

\begin{lemma} \label{lem:SwitchingTransverse}

Let $\Kc = \mathrm{span}\{e_1,\ldots, e_{d-1}\}$. For every $0 < \delta \ll 1$ there exists some $J_\delta \in \N$ and $c_\delta > 0$ with the following property. For every $x \in \Kd$ there exists a Borel measurable set $\mathcal{S}_x \subseteq \mathcal{S}$ with $m_{\mathcal{S}}(\mathcal{S}_x) \ge 1/2$ and such that for every $b \in \mathcal{S}_x$ we have
\begin{equation}
	\left| \Pi_\mathcal{K}^\perp \frac{d^j}{dt^j}\varphi_b(t,x)_{t=0} \right| \ge c_\delta
\end{equation}
for some $j \le J_\delta$.
\end{lemma}

\begin{proof}
	
We begin by using a transversality argument to prove that for each fixed $x \in \Kd$ it holds that for almost every $b \in \mathcal{S}$ there exists some $j \in \N$ such that 
\begin{equation} \label{eq:modifiedtransversality}
	\Pi_{\Kc}^\perp \frac{d^j}{dt^j} \varphi_b(t,x)|_{t=0} \neq 0.
\end{equation}
Fix $x \in K_\delta$. For $t_0 \in (0,1)$ to be chosen sufficiently small, define the map $F:(t_0,2t_0) \times \mathcal{S} \to \R^d$ by $F(t,b) = \varphi_b(t,x)$. We will show that $F$ is transversal to $\mathcal{K} \subseteq \R^d$. To prove this, we must show that for any $(\bar{t}, \bar{b}) \in (t_0,2t_0) \times \mathcal{S}$ such that $F(\bar{t},\bar{b}) \in \mathcal{K}$, there exists some $a \in T_{\bar{b}}\mathcal{S}$ such that $\Pi_{\mathcal{K}}^\perp[D_b F(\bar{t},\bar{b})](a) \neq 0$. For any $a \in T_{\bar{b}}\mathcal{S}$, $t \mapsto [D_b \varphi_{\bar{b}}(t,x)](a)$ solves
 \begin{equation} \label{eq:DbSystem1}
 	\begin{cases}
 		\frac{d}{dt} [D_b \varphi_{\bar{b}}(t,x)](a) = B_a(\varphi_{\bar{b}}(t,x),\varphi_{\bar{b}}(t,x)) + 2 B_{\bar{b}}([D_b \varphi_{\bar{b}}(t,x)](a),\varphi_{\bar{b}}(t,x)), \\ 
 		[D_b \varphi_{\bar{b}}(0,x)](a) = 0.
 	\end{cases}
 \end{equation} 
It follows then by Gr\"{o}nwall's lemma that 
 \begin{equation} \label{eq:DbSystem2}
 	\sup_{0 \le t \le 2t_0}|[D_b \varphi_{\bar{b}}(t,x)](a)| \leqc  t_0.
 	\end{equation}
 It is also easy to see that 
 \begin{equation} \label{eq:DbSystem3}
 	\sup_{0 \le t \le 2t_0}|B_a(\varphi_{\bar{b}}(t,x),\varphi_{\bar{b}}(t,x)) - B_a(x,x)| \leqc t_0.
 \end{equation}
Integrating \eqref{eq:DbSystem1} and then using \eqref{eq:DbSystem2} and \eqref{eq:DbSystem3} we find
\begin{equation}\label{eq:DbSystem4}
	|[D_b \varphi_{\bar{b}}(\bar{t},x)](a) - \bar{t} B_a(x,x)| \leqc t_0^2.
\end{equation}
Since $x \in K_\delta$, either $|x_d| \gtrsim \delta$ or there exist $1\le i < j < d$ such that $|x_i|,|x_j| \gtrsim \delta$. In the first case, there exists $0 < t_1 \ll \delta$ such that $F$ and $\mathcal{K}$ can have no intersection for $t_0 \le t_1$. We may thus assume the latter case. We choose then $a^d_{ij} = 1$ and all other elements zero, which gives $|\Pi_{\mathcal{K}}^\perp B_a(x,x)| \gtrsim \delta^2$. As this choice of $a$ and the implicit constant in \eqref{eq:DbSystem4} are both independent of $(\bar{t}, \bar{b}) \in (t_0, 2t_0) \times \mathcal{S}$, it follows from \eqref{eq:DbSystem4} that there exist constants $C, c > 0$ that do not depend on $\bar{b}$ or $\bar{t}$ such that
$$|\Pi_{\Kc}^\perp D_b F(\bar{t}, \bar{b})| =  |\Pi_{\mathcal{K}}^\perp[D_b \varphi_{\bar{b}}(\bar{t},x)](a)| \ge c \delta^2 t_0 - Ct_0^2.$$
Therefore, $F$ is transversal to $\Kc$ provided that $t_0 < \min(t_1,c\delta^2/2C)$. By the Transversality Theorem, it follows that for almost every $b \in \mathcal{S}$, if $t \in (t_0,2t_0)$ is such that $\varphi_b(t,x) \in \mathcal{K}$, then 
$$ \Pi_{\mathcal{K}}^\perp\frac{d}{dt} \varphi_b(t,x) \neq 0.$$
Thus, for almost every $b \in \mathcal{S}$, $t \mapsto \Pi_{\mathcal{K}}^\perp \varphi_b(t,x)$ is not identically zero on the time interval $[0,2t_0]$, and hence since $t \mapsto \Pi_{\mathcal{K}}^\perp \varphi_b(t,x)$ is analytic its zeros must be isolated. The claim of \eqref{eq:modifiedtransversality} then follows by Taylor expanding at $t = 0$.

With \eqref{eq:modifiedtransversality} established, the lemma follows easily. For $x \in \Kd$, $\theta > 0$, and $J \in \N$, let $\mathcal{S}_{J,\theta}(x) \subseteq \mathcal{S}$ denote the set of $b \in \mathcal{S}$ such that 
$$ \left| \Pi_{\mathcal{K}}^\perp \frac{d^j}{dt^j} \varphi_b(t,x)\big|_{t=0}\right| \ge \theta$$
for some $0 \le j \le J$. By \eqref{eq:modifiedtransversality} and the continuity of measure, for any $x \in \Kd$ there exists some $J_x \in \N$ and $\theta_x > 0$ such that $m_{\mathcal{S}}(\mathcal{S}_{J_x,\theta_x}(x)) > 1/2$. For a fixed $j \in \N$, the function $x \mapsto \frac{d^j}{dt^j} \varphi_b(t,x)\big|_{t=0}$ is Lipschitz continuous uniformly in $b \in \mathcal{S}$. This implies that for any $x \in \Kd$ there exists some $\epsilon_x > 0$ such that $m_{\mathcal{S}}(\mathcal{S}_{J_x,\frac{\theta_x}{2}}(y)) > 1/2$ for all $|y-x|< \epsilon_x$. The lemma now follows from the compactness of $\Kd$. 
\end{proof}

We now turn to our statement describing how solutions of \eqref{eq:SDE} spend sufficient time away from $\Ec$. We first record a version of Lemma~\ref{lem:scaledexit} that modifies slightly the set $\Kd$ and makes clear that the relevant constants are uniform in $b \in \mathcal{S}$. In what follows, we write $x_t^{(\epsilon,b)}$ for the solution of \eqref{eq:scaledSDE} with $B = B_b$ and $P_t^{(\epsilon,b)}$ for the associated Markov semigroup.

\begin{lemma} \label{lem:uniformstopping}
For $\delta \in (0,1)$ and $b \in \Bc$, let 
$$\tilde{\Kd} = \{x \in \R^d: \mathrm{dist}(x,\Ec) \ge \delta \text{ and } 3/4 \le |x| \le 5/4\}$$
and 
$$\tau_\delta^{(\epsilon,b)}(x_0, \omega) = \inf\{t \ge 0: x_t^{(\epsilon,b)} \in \tilde{\Kd} \}$$
There exist $\delta, c, C > 0$ such that for all $\epsilon$ sufficiently small, $b \in \mathcal{S}$, and $x_0 \in  \{4/5 \le |x| \le 6/5\}\setminus \tilde{\Kd}$ we have 
\begin{equation}
	\P(\tau_\delta^{(\epsilon,b)}(x_0, \omega) \le C|\log \epsilon|) \ge c.
\end{equation}
\end{lemma}

\begin{proof} 
	It is clear that using $\tilde{\Kd}$ instead of $\Kd$ and allowing for initial conditions off of $\S^{d-1}$ is not important, and so the proof of this lemma amounts to checking that the constants $c,C,\delta > 0$ and choice of $\epsilon$ sufficiently small in Lemma~\ref{lem:scaledexit} can be chosen uniformly over $b \in \mathcal{S}$. The dependence of the constants in Lemma~\ref{lem:scaledexit} on the nonlinearity comes only from upper bounds for $B$ and its derivatives, a lower bound on the spectral gaps of the matrices $L_{\bar{x}}$ for each $\bar{x} \in \Ec \cap \S^{d-1}$, an upper bound on the change of basis matrix (and its inverse) to the real Jordan canonical form of each $L_{\bar{x}}$, and the constants in Lemma~\ref{lem:generalhormander} applied with $X = B$ and $\tilde{X} = -Ax$. That bounds on $B_b$ and its derivatives are uniform over $b \in \mathcal{S}$ is immediate from the fact that $\mathcal{S}$ is bounded. Uniform-in-$b$ bounds on the spectral gap and norm of the relevant change of basis matrix for each $\bar{x} \in \Ec \cap \S^{d-1}$ follow from well-known facts in spectral theory and that  $\mathcal{S}$ is compactly contained in $\mathring{\Bc}$.\footnote{For $b \in \Bc$ and $\bar{x} \in \Ec \cap \S^{d-1}$, let $L_{\bar{x}}^{b}$ denote the linearization of $B_b$ at the fixed point $\bar{x}$. The proof of generic hyperbolicity from Section~\ref{subsec:genercHyperbolicity3} shows that the spectrum of $L_{\bar{x}}^b$ is purely simple for every $b \in \mathring{\Bc}$ . Therefore (see e.g. \cite[Chapter 2.1.1]{kato2013perturbation}), for each $b_0 \in \mathring{\Bc}$ there exists a closed ball $\mathcal{U}_{b_0} \subseteq \mathring{\Bc}$ containing $b_0$ such that for $b \in \mathcal{U}_{b_0}$ the dimensions of the stable and unstable subspaces of $L_{\bar{x}}^b$ are constant and both the eigenvalues and eigenvectors of $L_{\bar{x}}^b$ vary continuously with respect to $b$. As the norm of the change of basis matrix and its inverse to the real Jordan canonical form depend only on the minimal angle between a complete set of real, linearly independent eigenvectors, it is clear that spectral gaps and change of basis matrices of the $L_{\bar{x}}^b$ are bounded uniformly over $\mathcal{U}_{b_0}$. Our required uniform bounds then follow by the compactness of $\mathcal{S}$.} Regarding the smoothing estimates, the proof of hypoellipticity given in Section~\ref{subsec:genericHypoellipticity} shows that the collection of vector fields $\{B_b - \epsilon Ax, \sigma_1 e_1, \ldots, \sigma_d e_d\}$ satisfies the parabolic H\"{o}rmander condition for every $\epsilon \in [0,1]$ and $b \in \mathring{\Bc}$. This is enough to imply that the constants $s$, $p$, and $C > 0$ in the functional inequality of Lemma~\ref{lem:generalhormander} (with the vector fields $X_{0,\epsilon} = B_b - \epsilon Ax$, $X_1 = \sigma_1 e_1, \ldots, X_d = \sigma_d e_d$) can be taken uniform over $\epsilon \in (0,1)$ and $b$ varying over compact subsets of $\mathring{\Bc}$.
	
\end{proof}

\begin{lemma} \label{lem:KdeltaTime}
There exists $\delta, c > 0$ so that for all $\epsilon$ sufficiently small, $b \in \mathcal{S}$, and $x_0 \in \S^{d-1}$ we have 
\begin{equation}
	\frac{1}{\epsilon^{-1}}\E \int_0^{\epsilon^{-1}} \mathbf{1}_{\Kd} (x_t^{(\epsilon,b)}) \dt \ge \frac{c}{|\log \epsilon|}.
\end{equation}
\end{lemma}

\begin{proof}
By Lemma~\ref{lem:uniformstopping}, there exist $\delta', c, C$ so that for all $\epsilon > 0$ sufficiently small, $b \in \mathcal{S}$, and $4/5 \le |x| \le 6/5$ we have $\P(\tau_{\delta'}^{(\epsilon,b)}(x,\omega) \le C|\log \epsilon|) \ge c$, where $\tau_{\delta'}^{(\epsilon,b)}$ is the first hitting time to $\tilde{\mathcal{K}_{\delta'}}$ as defined in Lemma~\ref{lem:uniformstopping}. Let $\delta = \delta'/2$, so that $\tilde{\mathcal{K}_{\delta'}} \subseteq \Kd$ and $\mathrm{dist}(\tilde{\mathcal{K}_{\delta'}}, \partial \Kd) \gtrsim \delta$. It is easy to check then that there are constants $c_1, c_2 > 0$ that do not depend on $b \in \mathcal{S}$ or $\epsilon \in (0,1)$ such that for any initial condition in $\tilde{\mathcal{K}_{\delta'}}$ we have 
\begin{equation} \label{eq:RemainOutside}
	\P(x_t^{(\epsilon,b)} \in \Kd \quad \forall 0 \le t \le c_1) \ge c_2.
\end{equation}
Let $T = C|\log \epsilon| + c_1$ and for $n \in \N$ define $T_n = nT$. Let $N$ be the largest natural number such that $N T \le \epsilon^{-1}$ and let $\mu_n^{(\epsilon,b)}$ denote the law of $x_{T_n}^{(\epsilon,b)}$. Then, we have 
\begin{equation} \label{eq:RemainOutside1}
	\epsilon \E  \int_0^{\epsilon^{-1}} \mathbf{1}_{\Kd} (x_t^{(\epsilon,b)}) \dt \ge \epsilon \sum_{n=0}^{N-1} \E \int_{T_n}^{T_{n+1}} \mathbf{1}_{\Kd} (x_t^{(\epsilon,b)}) \dt = \epsilon \sum_{n=0}^{N-1}\int_{\R^d}\left[\int_0^T (P_t^{(\epsilon,b)} \mathbf{1}_{\Kd})(x)\dt \right] \mu_n^{(\epsilon,b)}(\dx).
\end{equation}
Using \eqref{eq:RemainOutside} and the bound on $\tau_{\delta'}^{(\epsilon,b)}$ provided by Lemma~\ref{lem:uniformstopping}, it is straightforward to apply the strong Markov property to show that for any $x \in \R^d$ with $4/5 \le |x| \le 6/5$ there holds 
\begin{equation}\label{eq:RemainOutside2}
	\int_0^T (P_t^{(\epsilon,b)} \mathbf{1}_{\Kd})(x) \dt \ge c c_1 c_2.
\end{equation}
By Remark~\ref{rem:energycorollary}, for $\epsilon$ taken sufficiently small independently of $b \in \mathcal{S}$ and $x_0 \in \S^{d-1}$, one has $\mu_n^{(\epsilon,b)}(4/5 \le |x| \le 6/5) \ge 1/2$ for every $n \le N$. Thus, combining \eqref{eq:RemainOutside1} and \eqref{eq:RemainOutside2} we get 
$$\epsilon \E \int_0^{\epsilon^{-1}}\mathbf{1}_{\Kd}(x_t^{(\epsilon,b)}) \dt \ge \epsilon N \frac{cc_1c_2}{2}.$$
Since $N \approx \epsilon^{-1}/|\log \epsilon|$, the proof is complete. 
\end{proof}

We are now ready to finish the proof of Theorem~\ref{thm:mainswitched}.

\begin{proof}[Proof of Theorem~\ref{thm:mainswitched}]
	As shown earlier, it suffices to verify hypothesis \eqref{eq:DiscreteKBgoal} of Lemma~\ref{lem:discreteexistence}. Let $\delta > 0$ be as in Lemma~\ref{lem:KdeltaTime} and for $R \ge 1$ define 
	$$ \mathcal{U}_R = \{x \in \R^d: x/R \in \Kd\}.$$
We first claim that there exist $c_1, C_1 > 0$ and $R_1 \gg 1$ such that for all $R \ge R_1$ we have
	\begin{equation} \label{eq:claim1final}
		\mathcal{P} V(x_0) \le - c_1 |x_0| + V(x_0) + C_1 \quad \forall x_0 \in \mathcal{U}_{R}.
	\end{equation} 
	Let $x_0 \in \mathcal{U}_R$, $\mathcal{S}_{x_0/R} \subseteq \mathcal{S}$ be the set guaranteed by Lemma~\ref{lem:SwitchingTransverse}, and $x_t^{b}$ denote the solution of \eqref{eq:SDE} with $B=B_b$ and initial condition $x_0$. Recalling \eqref{eq:discretesemigroup} and \eqref{eq:Ptbenergy}, we have
	\begin{equation}\label{eq:finalproof1}
		\mathcal{P}V(x_0) \le (1-m_{\mathcal{S}}(\mathcal{S}_{x_0/R}))(V(x_0) + C) + \int_I \int_{\mathcal{S}_{x_0/R}} P_t^b V(x_0) m_{\mathcal{S}}(\dee b) \dt.
	\end{equation}
	With the goal of bounding the second term in \eqref{eq:finalproof1}, we now estimate $P_t^b V(x_0)$ for $(t,b) \in I\times \mathcal{S}_{x_0/R}$. First, we apply the energy inequality \eqref{eq:energyestimate} and rewrite the dissipation term $\E \int_0^t A x_s^b \cdot x_s^b \dee s$ in terms of $x_s^{(\epsilon,b)} = \epsilon x_{\epsilon s}^b$ with $\epsilon = 1/R$ to get 
	\begin{equation} \label{eq:PtbV1}
		P_t^b V(x_0) = V(x_0) - \frac{2}{\epsilon} \E \int_0^{t/\epsilon} Ax_s^{(\epsilon,b)} \cdot x_s^{(\epsilon,b)}\dee s + t \sum_{i=1}^{d} \sigma_i^2.
		\end{equation}
	Since $x_0^{(\epsilon,b)} = x_0/R \in \Kd$, $b \in \mathcal{S}_{x_0/R}$, $t \ge 1/2$ and $1/\epsilon = R \ge 2|x_0|/3$, it follows by Lemma~\ref{lem:SwitchingTransverse} and Lemma~\ref{lem:conservativediss} that for $R \gg 1$ independent of $b$ there holds
	\begin{equation} \label{eq:PtbV2}
		- \frac{2}{\epsilon} \E \int_0^{t/\epsilon} Ax_s^{(\epsilon,b)} \cdot x_s^{(\epsilon,b)}\dee s \le -\frac{4|x_0|}{3} \E \int_0^1 Ax_s^{(\epsilon,b)} \cdot x_s^{(\epsilon,b)}\dee s \le - c|x_0|,
	\end{equation}
	where $c$ is a constant that does not depend on $x_0$ or $b$. We noted here that the constant $c_\delta$ and smallness requirement on Lemma~\ref{lem:conservativediss} applied with $B = B_b$ are both uniform in $b \in \mathcal{S}$ by the analysis in \cite[Section 3]{BedrossianLiss22}. Putting \eqref{eq:PtbV2} into \eqref{eq:PtbV1}, using the resulting bound in \eqref{eq:finalproof1}, and lastly recalling that $m_{\mathcal{S}}(\mathcal{S}_{x_0/R}) \ge 1/2$ gives \eqref{eq:claim1final}.
	
	We now use \eqref{eq:claim1final} and Lemma~\ref{lem:KdeltaTime} to complete the proof. Fix $x_0 \in \R^d$ with $|x_0| \ge R$ for $R \ge R_1$ to be taken sufficiently large. Let $\{\Phi_n\}_{n \in \N}$ denote the Markov chain recalled at the beginning of this section initiated at $x_0$, and let $\nu$ denote the law of $\Phi_1$. Using that $\E \mathbf{1}_{\Kd}(x_t^{(\epsilon,b)}) = \P(x_t^{(\epsilon,b)} \in \Kd)$, it is straightforward to show with Lemma~\ref{lem:KdeltaTime}, a rescaling argument, and Chebyshev's inequality that there is a constant $c_2$ that does not depend on $x_0$ such that 
	\begin{equation} \label{eq:finalproof3}
		\nu(\mathcal{U}_{|x_0|}) \ge \frac{c_2}{\log(|x_0|)}
	\end{equation}
	for all $R$ sufficiently large. By  \eqref{eq:finalproof3}, \eqref{eq:Ptbenergy}, \eqref{eq:claim1final} and the fact that 
	\begin{equation} \label{eq:finalproof4}
		\mathcal{P}^2 V(x_0) = \int_{\R^d} \mathcal{P}V(x) \nu(\dx),
	\end{equation}
	we have 
	\begin{align*}
		\mathcal{P}^2 V(x_0) & = \int_{\mathcal{U}_{|x_0|}} \mathcal{P} V(x) \nu(\dx) + \int_{\mathcal{U}^c_{|x_0|}} \mathcal{P} V(x) \nu(\dx) \\ 
		& \le \int_{\mathcal{U}_{|x_0|}} (-c_1 |x| + V(x) + C_1)\nu(\dx) + \int_{\mathcal{U}_{|x_0|}^c} (V(x) + C) \nu(\dx) \\ 
		& = C + C_1 + \mathcal{P} V(x_0) - c_1 \int_{\mathcal{U}_{|x_0|}} |x| \nu(\dx) \\ 
		& \le V(x_0) + 2C + C_1 - \frac{c_1 |x_0|}{2} \nu(\mathcal{U}_{|x_0|}) \\ 
		& \le V(x_0) + 2C + C_1 - \frac{c_1 c_2 |x_0|}{2\log(|x_0|)}.  
	\end{align*}
	Using the negative term in the final line above to absorb the contribution from $2C + C_1$ for $R$ sufficiently large completes the proof.
\end{proof}

\addcontentsline{toc}{section}{References}
\bibliographystyle{abbrv}
\bibliography{Biblio}

\end{document}